\documentclass[a4paper,10pt]{article}
\usepackage[a4paper,top=3cm,bottom=3cm,left=3cm,right=3cm,marginparwidth=4cm]{geometry}

\usepackage[utf8]{inputenc}

\usepackage{amsmath}
\usepackage{amssymb}
\usepackage{amsfonts}
\usepackage{amsthm}
\usepackage{mathtools}
\usepackage{graphicx}
\usepackage{booktabs}
\usepackage{hhline}
\usepackage[colorinlistoftodos]{todonotes}
\usepackage{soul}
\usepackage[ruled, linesnumbered]{algorithm2e}
\usepackage{hyperref}
\usepackage{bm}
\usepackage{authblk}
\usepackage{cleveref}

\PassOptionsToPackage{numbers, sort&compress}{natbib}
\usepackage{natbib}

\newtheorem{theorem}{Theorem}[section]
\newtheorem{corollary}[theorem]{Corollary}
\newtheorem{lemma}[theorem]{Lemma}

\newtheorem{conjecture}[theorem]{Conjecture}

\newtheorem{observation}[theorem]{Observation}
\newtheorem{remark}[theorem]{Remark}

\newtheoremstyle{TheoremNum}
	{\topsep}{\topsep}              
	{\itshape}                      
	{}                              
	{\bfseries}                     
	{.}                             
	{ }                             
	{\thmname{#1}\thmnote{ \bfseries #3}}
\theoremstyle{TheoremNum}
\newtheorem{thmn}{Theorem}

\theoremstyle{definition}
\newtheorem{definition}[theorem]{Definition}

\usepackage{tikz}
\usepackage{caption, subcaption}

\usepackage{graphics}
\usepackage{array}

\newcommand{\R}{\mathbb{R}}

\newcommand{\cC}{\mathcal{C}}

\newcommand{\cT}{\mathcal{T}}
\newcommand{\cTC}{\mathcal{TC}}
\newcommand{\cOr}{\mathcal{OR}}
\newcommand{\cTB}{\mathcal{TB}}

\newcommand{\dpVC}{\textsc{Degree-3 Vertex Cover}}
\newcommand{\dpLDc}{\textsc{$L_\cC$-Distance}}
\newcommand{\dpAsDc}{\textsc{$A^*_\cC$-Distance}}
\newcommand{\dpADc}{\textsc{$A_\cC$-Distance}}
\newcommand{\dpLD}{\textsc{$L_\cOr$-Distance}}

\newcommand{\dpAsD}{\textsc{$A^*_\cOr$-Distance}}

\definecolor{lightgreen}{RGB}{180,255,150}
\definecolor{darkgreen}{RGB}{46,140,0}

\newcommand{\leo}[1]{#1}
\newcommand{\leonew}[1]{\textcolor{black}{#1}}

\newcommand{\ym}[1]{\textcolor{black}{#1}}
\newcommand{\markj}[1]{#1}
\newcommand{\ej}[1]{\textcolor{black}{#1}}

\DeclareMathOperator{\head}{head}
\DeclareMathOperator{\tail}{tail}
\DeclareMathOperator{\Gad}{Gad}
\DeclarePairedDelimiter\floor{\lfloor}{\rfloor}
\DeclarePairedDelimiter\ceil{\lceil}{\rceil}

\definecolor{lightblue}{RGB}{180,180,255}

\title{Proximity Measures for Classes of Phylogenetic Networks\footnote{This research was partially funded by the Dutch Research Council (NWO) grant OCENW.KLEIN.125, OCENW.M.21.306, and OCENW.GROOT.2019.015.}}
\author[1]{Leo van Iersel\thanks{L.J.J.vanIersel@tudelft.nl}}
\author[1]{Mark Jones\thanks{m.jones@mdx.ac.uk}}
\author[1]{Esther Julien\thanks{E.A.T.Julien@tudelft.nl}}
\author[2]{Yangjing Long\thanks{yangjing@ccnu.edu.cn}}
\author[1]{Yukihiro Murakami\thanks{Y.Murakami@tudelft.nl}}

\affil[1]{Delft Institute of Applied Mathematics, TU Delft, the Netherlands}
\affil[2]{School of Mathematics and Statistics, Central China Normal University, China}

\date{\today}

\begin{document}

\maketitle

\begin{abstract}
\ym{Phylogenetic networks are used to represent the evolutionary history of species.
Due to biological interpretations and computational advantages, researchers have focused on restricted classes of phylogenetic networks, such as tree-child, orchard, and tree-based.
These classes capture different notions of tree-likeness: tree-child networks require every internal vertex to have a taxon reachable by a tree path, orchard networks are trees with horizontal arcs (for modelling histories rife with horizontal gene transfers), and tree-based networks are trees with additional (not-necessarily horizontal) arcs.
A natural question to ask is ``how far is a given network from belonging to a particular class?'' 
This motivates the study of proximity measures, which measure the minimum number of graph modifications required to transform a network into one belonging to a particular class.
In this paper, we consider three proximity measures based on leaf addition, valid arc deletion, and arc deletion.
We study pairwise comparability of the proximity measures, prove complexity results, and derive extremal bounds for the classes of tree, tree-child, orchard, and tree-based networks. 
}
\end{abstract}

\section{Introduction}\label{sec:intro}
Phylogenetic trees are used to represent the evolutionary history of species. While they are effective for illustrating speciation events through vertical descent, they are insufficient in representing more intricate evolutionary processes. Reticulate (net-like) events such as hybridization and horizontal gene transfer (HGT) can give rise to signals that cannot be represented on a single tree~\cite{goulet2017hybridization,wickell2020evolutionary}. In light of this, phylogenetic networks have received increasing attention due to their capability in elucidating reticulate evolutionary processes. 

\leonew{Since the space of all phylogenetic networks is huge and contains extremely complex and possibly unrealistic evolutionary histories,} phylogenetic networks are often categorized into different classes based on their topological features. These are often motivated computationally, but some classes are also defined based on their biological relevance~\cite{pardi2015reconstructible}. Classical examples of network classes involve the \emph{tree-child networks}~\cite{cardona2008comparison} and the \emph{tree-based networks}~\cite{francis2015phylogenetic}.
Roughly speaking, tree-child networks are those where every vertex has passed on a gene via vertical descent to an extant species, and tree-based networks are those obtainable from a tree by adding so-called \emph{linking arcs} between tree arcs. Recent developments have culminated in the introduction of \emph{orchard networks}, which lie -- inclusion-wise -- between the two aforementioned network classes~\cite{janssen2021cherry,erdHos2019class}. 
The class has been shown to be both algorithmically attractive and biologically relevant; they are defined as networks that can be reduced to a single leaf by a series of so-called \emph{cherry-picking operations}, and they were shown to be networks that can be obtained by adding horizontal arcs to trees (where the tree is drawn with the root at the top and arcs pointing downwards)~\cite{van2022orchard}. Such horizontal arcs can be used to model HGT events, making orchard networks especially apt in representing evolutionary scenarios where every reticulate event is a horizontal transfer. 

\leonew{However, none of these classes are exhaustive in the sense that all realistic evolutionary histories are necessarily contained in the class. Proximity measures were introduced to measure how far a given network is from belonging to a certain class~\cite{francis2018new}. Here we consider proximity measures based on certain graph operations. For example, it is well-known that any phylogenetic network can be transformed into an orchard network by adding leaves. These leaves can be interpreted as extinct, undiscovered or ignored species. One of the questions studied in this paper is: how many leaves need to be added to a given network to make it an orchard network?}


In this paper, we study four \leonew{network classes}: trees (i.e., networks without reticulations), tree-child networks, orchard networks, and tree-based networks. We determine the \textit{distance} of
\leonew{a given} network to each of these network classes via the following graph 
\leonew{operations}:
\textit {leaf addition}, \textit{arc deletion} \leonew{and \textit{valid arc deletion}}. Adding a leaf is the action of \leonew{subdividing an arc by a new vertex and adding a new arc from the subdividing vertex to a new leaf.}
\leonew{Valid arc deletion}
is the action of \leonew{deleting an arc and suppressing its endpoints such that the resulting graph is a valid phylogenetic network (i.e. it has no unlabelled outdegree-0 vertices and no parallel arcs).
An arc deletion is the action of deleting an arbitrary reticulation arc, and transforming the resulting graph into a phylogenetic network by repeatedly deleting unlabelled outdegree-$0$ vertices and suppressing parallel arcs and indegree-1 outdegree-1 vertices.}
From a biological standpoint, these measures could be helpful to indicate how many taxa are unsampled (leaf addition), or
\leonew{to account for noisy data}
(arc deletion), when one assumes the evolutionary events should be 
\leonew{explainable}
by a certain type of network class. For instance, if one hypothesizes that all reticulations are 
caused by HGT events, and the network is non-orchard, it could be helpful to study how many additional unsampled taxa ought to be added to the dataset such that all reticulate events 
\leonew{can be explained as HGT events}. Similarly, it would be interesting to investigate which arcs cause the network to be non-orchard. The arc deletion measure
\leonew{could be used to indicate which arcs are possibly due to noise in the data.}
For instance, if the network is constructed using a multilocus tree data set, the trees containing the removed arcs may be \leonew{(partly)} incorrect.
Given that HGT is the primary driver of reticulate evolution in bacteria~\cite{gyles2014horizontally}, determining the leaf addition and valid arc deletion distances is a vital inquiry. To illustrate, we provide a network of a few fungi species in \Cref{fig:fungi_added_leaf} which requires one additional leaf to make it orchard. 
\begin{figure}
    \centering
    \includegraphics{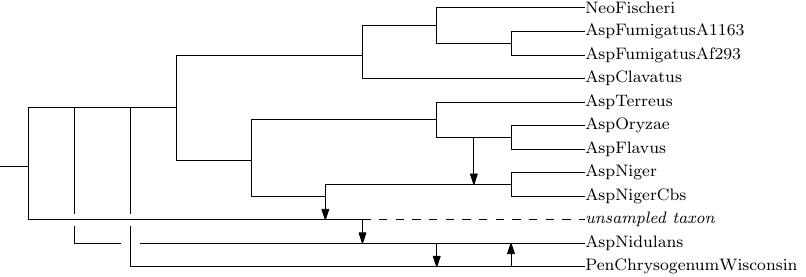}
    \caption{A network on 11 different \leo{taxa} (excluding the \emph{unsampled taxon}) of fungi \leo{including 5 reticulations, which is part of a larger network} from~\cite{szollHosi2015genome}. The directed arcs in the figure are linking arcs, \leo{which represent gene transfer highways}.
    \leo{In order to make all linking arcs horizontal}, we require an additional leaf (\textit{unsampled taxon}) to represent the evolutionary history. 
    }
    \label{fig:fungi_added_leaf}
\end{figure}

\leonew{There have been several previous papers considering proximity measures for phylogenetic network classes.} 
It was shown in~\cite{francis2018new} that networks can be made tree-based by adding leaves; the minimum number of leaves required to do so can be computed in polynomial time~\cite{hayamizu2021structure}.
In the former paper, this leaf addition measure was shown to be equivalent to two other proximity measures, which are based on spanning trees and disjoint path partitions. 
For orchard networks, a recent BSc thesis compared how the leaf addition proximity measure differs in general to another proximity measure based on arc deletions~\cite{susanna2022making}. 
\ym{This was explored further in a conference paper~\cite{van2023making}, of which this paper is an extended version. 
In that conference paper, the focus was on leaf addition measures for the classes of tree-child, orchard, and tree-based networks, with the main result being that computing the minimum leaf addition measure is NP-hard for orchard networks.
Here, we additionally consider arc deletion and valid arc deletion measures, study the comparability of pairwise measures for several network classes, and prove that computing the minimum valid arc deletion measure is NP-hard. 
The latter proof uses the same reduction as for the NP-hardness result for the leaf addition measure.
This final result also 
appeared recently in \cite{li2026computing}, which was obtained independently.
The shared author joined the preparation of this manuscript 
after these results had been obtained. 
}

The same question was posed for the unrooted variant (where the arcs of the network are undirected), for which the problem turned out to be NP-complete~\cite{fischer2020tree}. A total of eight proximity measures were introduced in this paper, including ones based on edge additions and rearrangement moves. 
In addition to leaf additions, leaf deletions (and, more generally, vertex deletions)
\leonew{have been considered for}
proximity measures for the class of so-called \emph{edge-based networks}~\cite{fischer2023far}. 

\paragraph{Our contributions.}
\ym{In this paper, we first attempt to answer the question ``are some proximity measures pairwise comparable?''
Oftentimes the answer is no (\Cref{thm:IncompLorA*or,thm:IncompLorAor,thm:IncompLTBATB}), though some measures are indeed comparable (\Cref{thm:ComparableLTBA*TB,thm:CompLTCATC}).
Then we look at each of the leaf addition and valid arc deletion measures separately. 
We show that the leaf addition measures are linear-time solvable for tree-child (\Cref{thm:L_TCisPoly}) and tree-based networks (shown in \cite[Corollary 5.4]{hayamizu2021structure}, but we include it here for completeness in \Cref{thm:L_TB=polynomial}).
We also give extremal bounds. 
Given a network with $r$ reticulations, we show that the leaf addition distances are at most $\lfloor (3r-1)/2 \rfloor$, $r-1$, and $\lfloor (r-1)/2\rfloor$, for tree-child, orchard, and tree-based networks, respectively, and show that these upper bounds are tight (\Cref{thm:L_TCBound}, \Cref{thm:LORUpperBound}, and \Cref{thm:L_TBBound}).
We show that computing this measure is NP-hard for orchard networks (\Cref{thm:L_Or=Hard}).
This same reduction is used to show that the valid arc deletion measure is NP-hard to compute for orchard networks~(\Cref{thm:A*_Or=Hard}).
For valid arc deletions, we also give a necessary and sufficient condition, formulated in terms of so-called \emph{one source one sink (OSOS) subgraphs}, for the proximity measure to the classes of trees, tree-child networks, and orchard networks to be finite (\Cref{cor:FiniteImpliesOSOSTCO}). 
We also give a polynomial-time algorithm to determine if a network contains an OSOS-subgraph, based on the Lengauer-Tarjan algorithm to find so-called \emph{dominator trees} (\Cref{thm:FindOSOSCorrect}). 
}

\paragraph{Structure.}
\ym{In \Cref{sec:preliminaries} we give relevant definitions and results from other works.
In \Cref{sec:ProximityMeasures}, we formally define the three proximity measures based on leaf addition, valid arc deletion, and arc deletion.
In \Cref{sec:comp}, we provide results on the comparability of the three proximity measures for certain network classes.
In \Cref{sec:LeafProximity}, we give complexity results and extremal bounds on the leaf addition measures.
In \Cref{sec:A*Deletion}, we determine when a network has finite valid arc deletion measure by considering OSOS-subgraphs.
In \Cref{sec:Hardness}, we show that computing the leaf addition and valid arc deletion measures for the class of orchard networks is NP-hard, by reduction from \textsc{Degree-3 Vertex Cover}.
In \Cref{sec:discussion}, we give concluding remarks and discuss potential future directions.}

\section{Preliminaries}
\label{sec:preliminaries}
A \emph{binary directed phylogenetic network} on a non-empty set~$X$ is a 
directed acyclic graph with
\begin{itemize}
    \item a single \emph{root} of indegree-0 and outdegree-1;
    \item \emph{tree vertices} of indegree-1 and outdegree-2;
    \item \emph{reticulations} of indegree-2 and outdegree-1;
    \item \emph{leaves} of indegree-1 and outdegree-0, that are labelled 
    bijectively by elements of~$X$.
\end{itemize}

For the sake of brevity, we shall refer to binary directed phylogenetic 
networks simply as \emph{networks}. Throughout the paper, we assume that~$N$ is a 
network on some non-empty set~$X$ where~$|X| = n$, unless stated otherwise.
Networks without reticulations are called \emph{trees}. Tree vertices and reticulations are  collectively referred to as \emph{internal vertices}.

The arc~$uv$ of a network is a \emph{root arc} if~$u$ is the root of the 
network. An arc~$uv$ of a network is a \emph{reticulation arc} if~$v$ is a reticulation, 
and a \emph{tree arc} otherwise. We say that a vertex~$u$ is a \emph{parent} of 
another vertex~$v$ if~$uv$ is an arc of the network; in such instances we 
call~$v$ a \emph{child} of~$u$. Also, we say that~$u$ and~$v$ are the 
\emph{tail} and the \emph{head} of the arc~$uv$, and denote them as~$\tail(uv)$ and~$\head(uv)$, respectively. 
In other words, 
we may rewrite arcs as~$uv = \tail(uv)\head(uv)$. The \emph{neighbours} of~$v$ 
refer to the set of vertices that are parents or children of~$v$.
We also say that vertices~$u$ and~$v$ are \emph{siblings} if they share the same parent.

In what follows, we shall define graph operations based on vertex and arc 
deletions. To make sure the resulting graphs are networks, we follow-up every 
graph operation with a \emph{cleaning up} process. Formally, we \emph{clean up} 
a network by applying the following until none is applicable.
\begin{itemize}
    \item Delete an unlabelled outdegree-0 vertex.
    \item Suppress an indegree-1 outdegree-1 vertex (e.g., if $uv$ and $vw$ are 
    arcs where~$v$ is an indegree-1 outdegree-1 vertex, we suppress~$v$ by 
    deleting the vertex~$v$ and adding an arc~$uw$).
    \item Replace parallel arcs by a single arc (e.g., if~$uv$ is an arc twice 
    in a network, 
    delete one of the arcs $uv$).
\end{itemize}
We observe that deleting a non-reticulation arc and cleaning up results in a 
graph containing two indegree-0 vertices. On the other hand, deleting a 
reticulation arc and cleaning up results in a network. Therefore, we shall use arc deletions to mean reticulation arc deletions. 
In the following, we introduce the different kinds of networks we study in this paper.


\subsection{Tree-Child Networks} \label{subsec:prelim_TC}
A network is \emph{tree-child} if every non-leaf vertex has a child that is a tree vertex or a leaf. 
We call an internal vertex of a network an \emph{omnian} if all of its children are reticulations~\cite{jetten2016nonbinary}. It follows from the definition that a network is tree-child if and only if it contains no omnians.

\subsection{Orchard Networks} \label{sec:orchard}





To define orchard networks, we must first define cherries and reticulated cherries, as well as operations to reduce them; see \Cref{fig:EGSequence} for the illustration of the following definitions.
Let~$N$ be a network. Two leaves~$x$ and~$y$ of~$N$ form a \emph{cherry} if they are siblings. In such a case, we say that~$N$ \emph{contains} (an ordered) cherry~$(x,y)$ or~$(y,x)$. Two leaves~$x$ and~$y$ of~$N$ form a \emph{reticulated cherry} if the parent~$p_x$ of~$x$ is a reticulation and the parent~$p_y$ of~$y$ is also a parent of~$p_x$. In such a case, we say that~$N$ \emph{contains} a reticulated cherry~$(x,y)$. \emph{Reducing the cherry~$(x,y)$ from~$N$} is the process of deleting the leaf~$x$ and cleaning up. \emph{Reducing the reticulated cherry~$(x,y)$ from~$N$} is the process of deleting the arc from the parent of~$y$ to the parent of~$x$, thus deleting $p_y p_x$, and cleaning up. In both cases, we use~$N(x,y)$ to denote the resulting network.

A network~$N$ is \emph{orchard} if there is a sequence~$S = (x_1,y_1)(x_2,y_2)\cdots(x_k,y_k)$
such that~$NS$ is a network 
on a single leaf~$y_k$. It has been shown that the order in which (reticulated) cherries are reduced does not matter~\cite{erdHos2019class,janssen2021cherry}.
Apart from this recursive definition, orchard networks have been characterized based on cherry covers (arc decompositions)~\cite{van2021unifying} and vertex 
labellings~\cite{van2022orchard}. We include both characterizations here.

\begin{figure}
    \centering
    \includegraphics{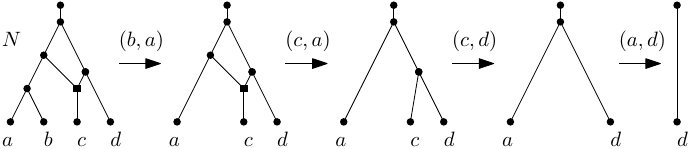}
    \caption{An example of an orchard network~$N$ that is reduced by a sequence $(b,a)(c,a)(c,d)(a,d)$.
    The network~$N$ contains a cherry $(b,a)$ and a reticulated cherry $(c,d)$.
    Subsequent networks are obtained by a single cherry-picking reduction from the previous network. For example, the second network $N(b,a)$ is obtained from $N$ by removing the leaf $b$ and cleaning up.}
    \label{fig:EGSequence}
\end{figure}

\paragraph{Cherry covers (see \cite{van2021unifying} for more details):} 
Let~$N$ be a network. 
A \emph{cherry shape} in~$N$ is a subgraph on three distinct vertices~$x, y, p$ with 
arcs~$px$ and~$py$. 
The \emph{internal vertex} of a cherry shape is $p$, and 
the \emph{endpoints} are~$x$ and~$y$. A \emph{reticulated cherry shape} in~$N$ is a 
subgraph on four distinct vertices~$x, y, p_x, p_y$ with arcs~$p_xx, p_yp_x, 
p_yy$, such that~$p_x$ is a  reticulation in the network~$N$. The internal vertices 
of a reticulated cherry shape are~$p_x$ and~$p_y$, and the endpoints are~$x$ 
and~$y$. 
The \emph{middle arc} of a 
reticulated cherry shape is
$p_yp_x$. 
We will often refer to (reticulated) cherry shapes by their arcs (e.g., we would denote the above cherry shape~$\{px, py\}$ and the reticulated 
cherry shape~$\{p_xx, p_yp_x, p_yy\}$). We say that an arc~$uv$ is covered by 
a cherry or reticulated cherry shape~$C$ if~$uv \in C$. 
A \emph{cherry cover} of a binary network is a set $P$ of cherry shapes and reticulated cherry shapes, such that each arc except for the root arc is covered exactly once by~$P$.
In general, a network can have more than one cherry cover.

We define the \emph{cherry cover auxiliary graph}~$G=(V,A)$ of a cherry cover~$P$ as follows. For each shape~$B \in P$, we have a vertex~$v_B \in V$ in the auxiliary graph $G$. A shape~$B \in P$ is \emph{directly above} another shape~$C \in P$ if $B$ and $C$ contain the same vertex $v$, such that $v$ is an endpoint of $B$ and an internal vertex of $C$. Then,~$v_Bv_C \in A$. (adapted from \cite[Definition 
2.13]{van2021unifying}). We say that a cherry cover is \emph{cyclic} if its auxiliary graph has a cycle. We call it \emph{acyclic} otherwise.
See \Cref{fig:cherry_cover_example} for an 
illustration of a cyclic and acyclic cherry cover. 

\begin{figure}
    \centering
        \subfloat[]{\includegraphics[height=4cm]{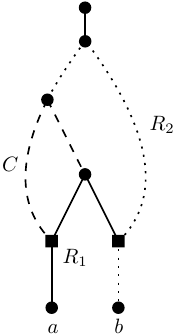} \label{subfig:cc_a}}
        \hspace{2mm}
        \subfloat[]{\includegraphics[height=1.2cm]{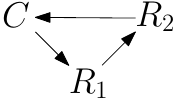} \label{subfig:cc_b}}
        \hspace{5mm}
        \subfloat[]{\includegraphics[height=4cm]{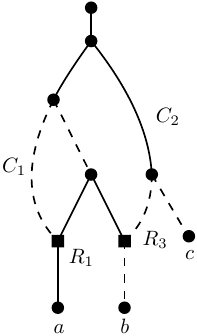} \label{subfig:cc_c}}
        \hspace{2mm}
        \subfloat[]{\includegraphics[height=1.2cm]{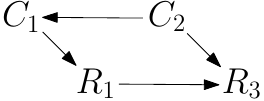} \label{subfig:cc_d}}
    \caption{An example of a cherry cover. 
    (a) A network~$N_1$ on~$\{a, b\}$ with a cherry cover~$\{C,R_1,R_2\}$.
    (b) The (cyclic) auxiliary graph of~$N_1$ based on the cherry cover of (a). 
    (c) The network~$N_2$ obtained from~$N_1$ by adding a leaf $c$, with a cherry cover~$\{C_1,C_2,R_1,R_3\}$
    (d) The (acyclic) auxiliary graph of~$N_2$ based on the cherry cover of (c).}
    \label{fig:cherry_cover_example}
\end{figure}

\begin{theorem}[Theorem 4.3 of 
\cite{van2021unifying}]\label{thm:AcyclicCherryCover}
    A network~$N$ is orchard if and only if it has an acyclic cherry cover.
\end{theorem}
\ej{By definition, a network is orchard if we can fully reduce it with a cherry-picking sequence. The auxiliary graph determines the order of picking cherries. If there is an arc from $C_1$ to $C_2$, it means that $C_2$ needs to be picked before $C_1$ can be picked. When all cherry-cover auxiliary graphs of a network are cyclic, it means there is no `terminal' cherry. Each cherry waits for another cherry to be picked, which is impossible. If one is terminal ($R_3$ in \Cref{subfig:cc_c,subfig:cc_d}), and there are no other cycles in the auxiliary graph, we can pick a cherry corresponding to one of the terminal nodes and reduce the other cherries as well.}

\paragraph{Non-Temporal Labellings (see \cite{van2022orchard} for more details):}\label{par:nontemporallabelling}

Let~$N$ be a network with vertex set~$V(N)$. A \emph{non-temporal 
labelling}\footnote{This is named in contrast to \emph{temporal 
representations} of~\cite{baroni2006hybrids}. There, it was required for the 
endpoints of every reticulation arc to have the same label.} 
of~$N$ is a labelling~$t:V(N)\rightarrow\R$ such that
\begin{itemize}
	\item for all arcs $uv$, $t(u)\le t(v)$ and equality is allowed only if~$v$ 
	is a reticulation;
	\item for each internal vertex~$u$, there is a child~$v$ of~$u$ such 
	that~$t(u) < t(v)$;
	\item for each reticulation~$r$ with parents~$u$ and~$v$, at most one 
	of~$t(u)=t(r)$ or~$t(v)=t(r)$ holds.
\end{itemize}
Observe that every network (orchard or not) admits a non-temporal labelling by labelling each vertex by its longest distance from the root (assuming each arc is of weight 1). 

Under non-temporal labellings, we call an arc \emph{horizontal} if its 
endpoints have the same label; we call an arc \emph{vertical} otherwise. By 
definition, only reticulation arcs can be horizontal. We say that a 
non-temporal labelling is an \emph{HGT-consistent labelling} if every 
reticulation is incident to exactly one incoming horizontal arc \cite[Definition 6]{van2022orchard}. We recall the following key result.

\begin{theorem}[Theorem 1 of~\cite{van2022orchard}]\label{thm:OrchIFFHori}
	A network is orchard if and only if it admits an HGT-consistent labelling.
\end{theorem}

Intuitively, \Cref{thm:OrchIFFHori} says that every orchard network is a phylogenetic tree with additional horizontal arcs. 
We return to the fungi network that we saw in \Cref{sec:intro}.

\begin{remark}\label{rem:NoAddedLeaf}

    We first elaborate on why we need an added leaf (\textit{unsampled taxon}) in the network of \Cref{fig:fungi_added_leaf} to ensure that the network admits an HGT-consistent labelling. We know that a network has an HGT-consistent labelling if and only if it is orchard (\Cref{thm:OrchIFFHori}). Let~$N$ be the network without \textit{unsampled taxon} (see~\Cref{fig:fungi_without_added_leaf}). We will show that~$N$ is not orchard. To see this, note that the order in which cherries and reticulated cherries are reduced does not matter \cite{janssen2021cherry}. This means that if~$N$ were orchard, then there would exist a cherry-picking sequence starting with
    \[(AspNidulans, PenChrysogenumWisconsin)(PenChrysogenumWisconsin, AspNidulans).\]
    After reducing these cherries, the distance between the leaf~$AspNidulans$ and any other leaf remains at a distance of at least~$4$, regardless of other reductions that take place in the network. This shows that the network cannot be orchard, and therefore the network cannot have an HGT-consistent labelling.
\end{remark}

\begin{figure}
    \centering\includegraphics{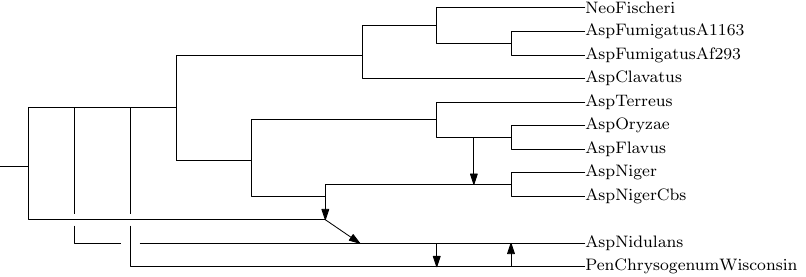}
    \caption{The network of \Cref{fig:fungi_added_leaf} without the added leaf. Observe that there exists no HGT-consistent labelling for the network, by the arguments provided in \Cref{rem:NoAddedLeaf}.}
    \label{fig:fungi_without_added_leaf}
\end{figure}

\subsection{Tree-Based Networks}

A network~$N$ is \emph{tree-based} with \emph{base tree}~$T$ if it can be 
obtained from~$T$ in the following steps \cite{francis2015phylogenetic}.
\begin{enumerate}
	\item Replace some arcs of~$T$ by paths, whose internal vertices we call 
	\emph{attachment points}; each attachment point is of indegree-1 and 
	outdegree-1.
	\item Place arcs between attachment points, called \emph{linking arcs}, so 
	that the graph contains no vertices of total degree greater than $3$, and 
	so that it remains acyclic.
	\item Clean up.
\end{enumerate}

We include here a static characterization of tree-based networks based on an arc partition, called \emph{maximum zig-zag 
trails}~\cite{hayamizu2021structure, zhang2016tree}. 
Let~$N$ be a network. A \emph{zig-zag trail} of length~$k$ is a sequence~$(a_1,a_2,\ldots, a_k)$ of arcs where~$k\ge 1$, and either~$\tail(a_i) = \tail(a_{i+1})$ 
or~$\head(a_i)=\head(a_{i+1})$ holds for~$i\in [k-1] = \{1,2,\ldots, k-1\}$. We 
call a zig-zag trail~$Z$ \emph{maximal} if there is no zig-zag trail that contains~$Z$ as a subsequence. Depending on the nature of~$\tail(a_1)$ and~$\tail(a_k)$, we have four possible maximal zig-zag trails. 
\begin{itemize}
	\item \emph{Crowns}: $k\ge 4$ is even and $\tail(a_1) = \tail(a_k)$ or 
	$\head(a_1) 
	= \head(a_k)$.
	\item \emph{M-fences}: $k\ge 2$ is even, it is not a crown, and $\tail(a_i)$ is a tree vertex for every~$i\in[k]$.
	\item \emph{N-fences}: $k\ge 1$ is odd and $\tail(a_1)$ or~$\tail(a_k)$, but not both, is 
	a reticulation. By reordering the arcs, assume henceforth that~$\tail(a_1)$ is a reticulation and~$\tail(a_k)$ a tree vertex.
	\item \emph{W-fences}: $k\ge 2$ is even and both~$\tail(a_1)$ 
	and~$\tail(a_k)$ are 
	reticulations.
\end{itemize}
We call a set~$S$ of maximal zig-zag trails a \emph{zig-zag decomposition} 
of~$N$ if the elements of~$S$ partition all arcs, except for the root arc, 
of~$N$.

\begin{theorem}[adapted from Theorem 4.2 of 
\cite{hayamizu2021structure}]\label{thm:UniqueMaxZigZag}
	Any network~$N$ has a unique zig-zag decomposition.
\end{theorem}

\ej{The following lemma gives the characterization based on the zig-zag decomposition of a network.}

\begin{lemma}[adapted from Corollary 4.6 of \cite{hayamizu2021structure}]\label{lem:TB=NoW}
    Let~$N$ be a network. Then~$N$ is tree-based if and only if it has no 
    W-fences.
\end{lemma}

\ej{The next theorem provides another characterization of tree-based networks, based on cherry covers that were introduced in the previous section. }
\begin{theorem}[adapted from Theorem 3.3 of \cite{van2021unifying}]\label{thm:TB=ChCover}
    Let~$N$ be a network. Then~$N$ is tree-based if and only if it has a cherry cover.
\end{theorem}
Recall that for orchard networks, the cherry cover has to be acyclic (\Cref{thm:AcyclicCherryCover}). This is not the case for tree-based networks, which implies that orchard networks are also tree-based.
One can also show that tree-child networks can be reduced via cherry-picking, thus showing that they are orchards.
We conclude the section with the following lemma.

\begin{lemma}[\cite{janssen2021cherry} and Corollary 1 of 
\cite{van2022orchard}]\label{lem:NetworkClassContainment}
	If a network is tree-child, then it is orchard. If a network is orchard, 
	then it is tree-based.
\end{lemma}

\section{Proximity Measures}\label{sec:ProximityMeasures}
In this paper, we consider \leonew{three} proximity measures to determine how much a 
network deviates from being in a certain network class. Such measures are 
defined on the basis of vertex additions and arc deletions. In this section, $\cC$ 
will be used to denote a network class. In particular, we shall use~$\cT, \cTC, 
\cOr$, and~$\cTB$ to denote the classes of trees, tree-child networks, orchard 
networks, and tree-based networks, respectively.

\paragraph{Leaf Addition.}
Let~$N$ be a network on~$X$. The first measure we consider is based on \emph{adding a leaf.} \emph{Adding a leaf~$x\notin X$ to an arc~$e$ 
of~$N$} is the process of adding a labelled vertex~$x$, subdividing the arc~$e$ by a vertex~$w$ (i.e., if~$e=uv$ we delete the arc~$uv$, add the vertex~$w$, and add arcs~$uw$ and~$wv$), and adding an arc~$wx$. We denote the resulting network by~$N+(e,x)$. When the arc~$e$ in the above is irrelevant, we simply call this process \emph{adding a leaf~$x$ to~$N$}, and denote the resulting network by~$N+x$. 

\begin{definition}[\textit{Leaf addition measure}] \label{def:leaf_add}
    Let~$L_{\cC}(N)$ denote the minimum number of leaf additions required to make the network~$N$ a member of~$\cC$. Adding leaves does not reduce the number of reticulations in a network, so it only makes sense to consider~$L_\cC(N)$ for~$\cC\in\{\cTC,\cOr,\cTB\}$. We refer to this measure as \emph{leaf addition}.
\end{definition}
We consider the following question regarding leaf addition.
\medskip
\noindent
\begin{center}
\fbox{
\parbox{0.9\linewidth}{
{\sc \dpLDc{}}\\
{\bf Input:} A network $N$ on a set of taxa $X$.\\
{\bf Question:} What is the minimum number of leaves that need to be added for $N$ to belong to network class $\cC$? In other words, what is~$L_\cC(N)$?}
}
\end{center}
In what follows, for the classes of tree-child and tree-based, leaves will be added to remove so-called `forbidden structures'.
For orchard networks, leaves will be added to create reticulated cherries (sequentially).
For all three classes, one adds at most a leaf to each arc in the original network.
Intuitively, this is because each arc can be in at most one forbidden structure or one reticulated cherry.

\paragraph{Valid Arc Deletion.}
The second measure is based on \emph{deleting arcs}. As stated in \Cref{sec:preliminaries}, we shall write arc deletions to mean reticulation arc deletions. \ej{Contrary to leaf additions, arc deletions are done sequentially}. Observe that upon deleting an arc, we are required to clean up its endpoints in the resulting graph. We call a reticulation arc~$e$ \emph{valid} if deleting~$e$ and cleaning up results in a network that has exactly three fewer arcs than that of the original network~\cite{murakami2019reconstructing}. 
The second measure is defined on the basis of valid arcs. 

\begin{definition}[\textit{Valid arc deletion measure}] \label{def:val_arc_del}
    Let~$A^*_{\cC}(N)$ denote the minimum number of valid arc deletions required to make the network~$N$ a member of~$\cC$. We define~$A^*_{\cC}(N)\coloneqq \infty$ when~$N\notin \cC$ and~$N$ has no sequence of valid arc deletions to result in a network in~$\cC$, with~$\cC \in \{\cT, \cTC, \cOr, \cTB\}$. We refer to this proximity measure as \emph{valid arc deletion}.
\end{definition}
We consider the following question regarding valid arc deletion.
\medskip
\noindent
\begin{center}
\fbox{\parbox{0.9\linewidth}{
{\sc \dpAsDc{}}\\
{\bf Input:} A network $N$ on $X$.\\
{\bf Question:} What is the minimum number of valid arcs that need to be deleted sequentially for $N$ to belong to network class $\cC$? In other words, what is~$A^*_\cC(N)$?}}
\end{center}
\medskip

To give intuition, we characterize invalid arcs.

\begin{lemma}\label{lem:ValidEdgeCharacterization}
	Let~$N$ be a network. A reticulation arc~$ur$ in~$N$ is invalid if and 
	only if one of the following holds.
	\begin{enumerate}
		\item $u$ is a reticulation; or
		\item $u$ is a tree vertex with parent~$p$ and children~$r,s$, where 
		$pu,us,ps$ are arcs in $N$; or
		\item $r$ has parents~$u,v$ and a child~$c$, and $vr,rc,vc$ are arcs 
		in~$N$; or
        \item $u$ is a tree vertex with parent~$p$ and children~$r,s$, where~$pu,us,pr,rs$ are arcs in~$N$.
	\end{enumerate}
\end{lemma} 

\begin{proof}
	If one of the four cases mentioned above is present in~$N$, then it is clear that upon deleting the arc~$ur$ and cleaning up, the resulting network contains at least four fewer arcs than~$N$.
    Then the arc~$ur$ must be invalid.

    So now suppose that~$ur$ is an invalid reticulation arc. We split into two cases depending on whether~$u$ is a reticulation or a tree vertex.
    If~$u$ is a reticulation, no other conditions are necessary to ensure that~$ur$ is invalid. Indeed, upon deleting~$ur$,~$u$ becomes an unlabelled vertex of indegree-2 and outdegree-0. 
    Upon cleaning up, we observe that~$u$ is deleted, along with its remaining incident arcs.
    All the while,~$r$ becomes a vertex of indegree-1 and outdegree-1. It therefore becomes suppressed in the cleaning up process. 
    In the resulting network, all arcs incident to~$u$ are removed, and in addition, an arc incident to~$r$ (but not to~$u$) is removed. 
    This means the resulting network contains at least four arcs fewer than that of the original network~$N$. This is case 1 in our lemma.

    On the other hand if~$u$ is a tree vertex, then the invalidity of~$ur$ implies that suppressing~$u$, ~$r$, or both~$u$ and~$r$ in the cleaning up process creates parallel arcs.
    First suppose that suppressing~$u$ creates parallel arcs. 
    Then we must have that~$u$ and the parent~$p$ of~$u$ share a common child~$s$, which is case 2 in our claim.
    Next suppose that suppressing~$r$ creates parallel arcs.
    Then, similar to the previous case, the parent~$v$ of~$r$ that is not~$u$ share a common child~$c$ with~$r$, which is case 3 in our claim.
    Finally, if suppressing both~$u$ and~$r$ creates parallel arcs, then~$u,r$ must have a common parent~$p$ and a common child~$s$, which is case 4 in our claim. 
\end{proof}

See \Cref{fig:example_Lor_Aorinf} (a) for an example where $A^*_{\cC}(N) = \infty$. None of the reticulated arcs $e_1, \ldots, e_4$ are valid: $e_2$ is incident to two reticulations and deleting any of the other arcs creates parallel arcs.

\paragraph{(Not necessarily valid) Arc Deletion.}
\ym{The third and final measure is based on \emph{deleting (not necessarily valid) arcs}.
As for the valid arc deletions, the arc deletions considered here are done sequentially. 
The third measure is defined on the basis of removing reticulation arcs.}

\begin{definition}[\textit{Arc deletion measure}] \label{def:arc_del}
    Let~$A_{\cC}(N)$ denote the minimum number of arc deletions required to make the network~$N$ a member of~$\cC$. We define~$A_{\cC}(N)\coloneqq \infty$ when~$N\notin \cC$ and~$N$ has no sequence of arc deletions to result in a network in~$\cC$, with~$\cC \in \{\cT, \cTC, \cOr, \cTB\}$. We refer to this proximity measure as \emph{arc deletion}.
\end{definition}
We consider the following question regarding arc deletion.
\medskip
\noindent
\begin{center}
\fbox{\parbox{0.9\linewidth}{
{\sc \dpADc{}}\\
{\bf Input:} A network $N$ on $X$.\\
{\bf Question:} What is the minimum number of arcs that need to be deleted sequentially for $N$ to belong to network class $\cC$? In other words, what is~$A_\cC(N)$?}}
\end{center}
\medskip

\ym{In what follows, we discuss the comparability of the three proximity measures in \Cref{sec:comp}.
Then, we study the leaf addition and valid arc deletion measures in detail in \Cref{sec:LeafProximity,sec:A*Deletion,sec:Hardness}, proving complexity results and bounds.}


\section{Comparability of Proximity Measures}\label{sec:comp}
In this section, we provide some results on the comparability of the three proximity measures. 
\ym{For the first one, we recall that trees are tree-child, tree-child networks are orchard, and orchard networks are tree-based (\Cref{lem:NetworkClassContainment}).
This gives rise to the following three observations.}

\begin{observation}\label{obs:LeafAdditionComparisonClasses}
    Let~$N$ be a network. Then,~$L_{\cTC}(N) \geq L_{\cOr}(N) \geq L_{\cTB}(N)$.
\end{observation}
\begin{observation}\label{obs:A*ClassIneq}
        Let $N$ be a network. Then, $A^*_{\cT}(N) \geq A^*_{\cTC}(N) \geq A^*_{\cOr}(N) \geq A^*_{\cTB}(N)$.
\end{observation}
\begin{observation}
    Let~$N$ be a network. Then, $A_{\cT}(N) \geq A_{\cTC}(N) \geq A_{\cOr}(N) \geq A_{\cTB}(N)$.
\end{observation}

\ym{We say that two measures~$M_1,M_2$ are \emph{comparable} if for any network~$N$, we have~$M_1(N) \le M_2(N)$.
We say that two measures~$M_1,M_2$ are \emph{incomparable} if there exist two networks~$N_1,N_2$, where~$M_1(N_1)>M_2(N_1)$ and~$M_1(N_2)<M_2(N_2)$.}

\subsection{\texorpdfstring{Comparing $A^*$ and~$A$}{Comparing A* and A}}\label{subsec:CompA^*A}
Since every valid arc deletion is an arc deletion, by definition,~$A$ and~$A^*$ must be comparable.
\begin{observation}\label{obs:AleA^*}
    Let~$\cC$ be a network class, and let~$N$ be a network. 
    Then~$A_\cC(N)\le A^*_\cC(N)$.
\end{observation}

\subsection{\texorpdfstring{Comparing~$L$ and~$A^*$}{Comparing L and A*}}\label{subsec:CompLA^*}

In this subsection, we show that the measures are incomparable and comparable for $\cOr$ and~$\cTB$, respectively.

\begin{figure}
    \centering
    \subfloat[]{\includegraphics{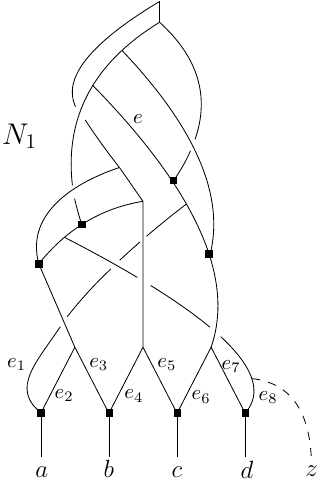}}
    \hspace{5mm}\label{subfig:L_orA_orN_1}
    \subfloat[]{\includegraphics{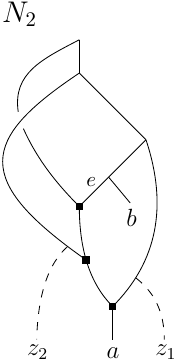}}\label{subfig:L_orA_orN_2}
    \caption{Two networks~$N_1$ and~$N_2$ on leaf sets~$\{a,b,c,d\}$ and~$\{a,b\}$, respectively, where $L_{\cOr}(N_1) <
    A^*_{\cOr}(N_1)$ and $L_{\cOr}(N_2) > A^*_{\cOr}(N_2)$. 
    We have~$L_{\cOr}(N_1)=1, L_{\cOr}(N_2)=2, A^*_{\cOr}(N_1) = A_{\cOr}(N_1) = 2$, and~$A^*_{\cOr}(N_2) = A_{\cOr}(N_2) = 1$. 
    See \Cref{thm:IncompLorA*or,thm:IncompLorAor} for the proofs thereof, which uses the leaves~$z,z_1$, and~$z_2$ to show the leaf addition measures of each network.
    }
    \label{fig:Lor_Aor_star_incomp}
\end{figure}

\begin{theorem}\label{thm:IncompLorA*or}
    The measures $L_{\cOr}$ and $A^*_{\cOr}$ are incomparable.
\end{theorem}

\begin{proof}
We prove the theorem by using the two networks~$N_1$ and~$N_2$ in \Cref{fig:Lor_Aor_star_incomp}. 
We show that~$L_{\cOr}(N_1) = 1$ and $A^*_{\cOr}(N_1) = 2$, and we show that~$L_{\cOr}(N_2) = 2$ and $A^*_{\cOr}(N_2) = 1$, thereby proving the claim. 

\paragraph{$\bm{L_{\cOr}(N_1) = 1}$:} First, observe that $N_1$ does not have any (reticulated) cherries, and hence at least one leaf should be added. 
Assume that we add a new leaf $z$ to arc $e_8$ as shown in the figure. 
Then, we can reduce the resulting network using the following cherry-picking sequence. $(d, z) (c, d) (b, c) (a, b) (b, z) (d, a) (z, c) (a, z) (z, d) (a, d) (b, c) (c, d)$. We conclude that $L_{\cOr}(N_1) = 1$.

\paragraph{$\bm{A^*_{\cOr}(N_1) = 2}$:} Similarly, as argued before, there are no (reticulated) cherries that can be picked, and so $A^*_{\cOr}(N_1) \geq 1$. 
To create at least one (reticulated) cherry, we must delete at least one of $e,e_1, \ldots, e_8$. 
We will check each case, and verify -- since the order of cherry reductions does not matter (\cite{janssen2021cherry}) -- that deleting just one of these arcs will not yield an orchard network. 
\begin{itemize}
    \item Delete $e$: There are no (reticulated) cherries.
    \item Delete $e_1$: We pick the elements $(b, a) (c, b) (d, c) (a, d) (d, b) (a, b)$. 
    In the resulting network, each leaf pair has a distance (when considering the underlying undirected graph) of at least 4 and can hence not be further reduced.
    \item Delete $e_2$: There are no (reticulated) cherries.
    \item Delete $e_3$: We pick the elements $(c, b) (d, c)$. There are no (reticulated) cherries left.
    \item Delete $e_4$: We pick the element $(a, b)$. There are no (reticulated) cherries left.
    \item Delete $e_5$: We pick the element $(d,c)$. There are no (reticulated) cherries left.
    \item Delete $e_6$: We pick the elements $(b, c)(a, b)$. There are no (reticulated) cherries left.
    \item Delete $e_7$: There are no (reticulated) cherries.
    \item Delete $e_8$: We pick the elements $(c, d) (b, c) (a, b) (d, a)$. There are no (reticulated) cherries left.
\end{itemize}
\ej{Now, suppose we delete $e$ and $e_1$. Then, reduce the resulting network by picking the elements $(b, a) (c, b) (d, c) (a, d) (d, b) (c, d) (a, b) (c, d) (b, d)$. This shows we can make $N_1$ orchard with two arc deletions and since $A^*_{\cOr}(N_1) > 1$, we conclude that $A^*_{\cOr}(N_1) = 2$.}

\paragraph{$\bm{L_{\cOr}(N_2) = 2}$:} The network $N_2$ on $\{a,b\}$ is not orchard; there is no (reticulated) cherry that can be picked, so~$L_{\cOr}(N_2) >0$.
We now show that $L_{\cOr}(N_2) > 1$.
Clearly, we must add leaves to create reticulated cherries.
No leaf addition in the neighbourhood of~$b$ will yield a reticulated cherry as~$b$'s parent is a tree vertex.
So one must add a leaf either to one of the two incoming arcs of the parent of~$a$.
If we add a leaf $l_1$ to the incoming arc of the parent of $a$ (in the figure, the left incoming arc), we can pick $(a, l_1)$, after which there is no (reticulated) cherry to pick. 
On the other hand, suppose we add a leaf $l_2$ (i.e., leaf $z_1$ in the figure) to the other incoming arc of the parent of $a$. 
After picking the element~$(a, l_2)$, there is no other (reticulated) cherry to pick. 
This means $L_{\cOr}(N_2) > 1$. 
Consider the network obtained by adding the two leaves $z_1$ and $z_2$ as indicated in \Cref{fig:Lor_Aor_star_incomp}. 
The resulting network can be reduced using the sequence $(a, z_1) (a, z_2) (a, b) (b, z_1), (z_1, z_2) (z_2, a)$.

\paragraph{$\bm{A^*_{\cOr}(N_2) = 1}$:} The network is not orchard, and hence at least one valid arc should be deleted. Consider the network obtained by removing~$e$. 
We can reduce the resulting network using the sequence $(a, b) (a, b) (a, b)$.\qedhere
\end{proof}



\begin{figure}
    \centering
    \subfloat[]{\includegraphics{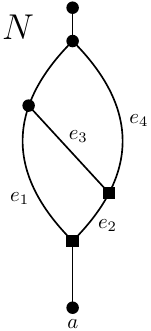}}
    \hspace{4mm}
    \subfloat[]{\includegraphics{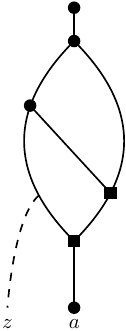}}
    \caption{
    Example of a network~$N$ with~$L_{\cOr}(N)<A^*_{\cOr}(N)$.
    (a) A non-orchard network $N$ on~$\{a\}$. 
    $N$ does not have any valid arcs, thus $A^*_{\cOr} = \infty$.
    (b) The network~$N$ with an additional leaf~$\{z\}$. The network can be reduced using the sequence~$(a,z)(a,z)(a,z)$, showing that $L_{\cOr}(N) = 1$.}
    \label{fig:example_Lor_Aorinf}
\end{figure}

When we look at the class of tree-based networks, we show that the leaf addition proximity measure is always at most the valid arc deletion measure.

\begin{figure}
    \centering
    \subfloat[]{\includegraphics{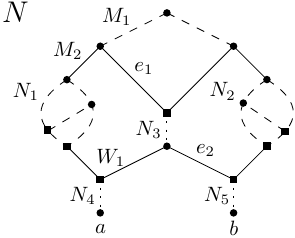}}
    \hspace{2mm}
    \subfloat[]{\includegraphics{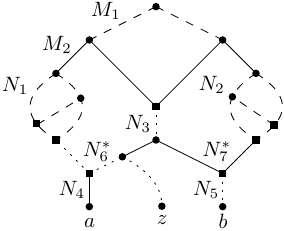}}
    \hspace{2mm}
    \subfloat[]{\includegraphics{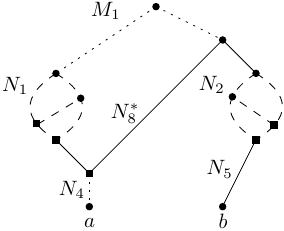}}
    \caption{
    Example of a network~$N$ with $1 = L_{\cTB}(N) < A^*_{\cTB}(N) = 2$.
    (a) A non-tree-based network~$N$ on~$\{a,b\}$, together with its zig-zag decomposition. 
    (b) $N$ with an additional leaf~$z$. The resulting network is tree-based as there are no more W-fences. This implies~$L_{\cTB}(N)=1$.
    (c) $N$ with valid arcs~$e_1,e_2$ removed. The resulting network is tree-based, and as one valid arc deletion does not suffice, $A^*_{\cTB}(N)=2$.}
    \label{fig:example_Ltb_l_Astb}
\end{figure}

\begin{theorem}\label{thm:ComparableLTBA*TB}
    The measures~$L_\cTB$ and~$A^*_\cTB$ are comparable.
    In particular, for all networks~$N$, $L_{\cTB}(N) \leq A^*_{\cTB}(N)$. 
\end{theorem}
\begin{proof}
    Let $A^*_{sol}$ be the sequence of optimal arc deletions to make $N$ tree-based.
    We prove by induction on~$|A^*_{sol}| = A^*_{\cTB}(N)$.
    For the base case, suppose~$A^*_{\cTB}(N)=1$, and let~$A^*_{sol}=\{uv\}$.
    Let $w_1u, ux_1, w_2v, vx_2 \in A(N)$ be arcs in~$N$. 
    After deleting $uv$, we suppress $u$ and $v$ and call the resulting network~$N'$.
    As~$N'$ is tree-based, it contains a base tree~$T$.
    Note that the arcs~$w_1x_1$ and~$w_2x_2$ are either arcs in~$T$ or not, in which case they are linking arcs.
    We distinguish the four possible cases, and show that~$L_{\cTB}(N)=1$ in each case.
    In each case, let~$p_z$ denote the parent of a newly added leaf~$z$.
    \begin{enumerate}
        \item $w_1x_1$ and $w_2x_2$ are both arcs in~$T$. But then the tree obtained by subdividing~$w_1x_1$ and~$w_2x_2$ by~$u$ and~$v$ respectively, is a base tree of~$N$. This contradicts the choice of~$N$, so this case cannot occur. 
        \item $w_1x_1$ is an arc in~$T$ and~$w_2x_2$ is a linking arc.
        We add a leaf~$z$ to the arc~$vx_2$ in~$N$. 
        The tree~$T$ with an additional path~$w_2vp_z z$ is a base tree of~$N$.
        \item $w_1x_1$ is a linking arc in~$T$ and~$w_2x_2$ is an arc in~$T$.
        We add a leaf~$z$ to the arc~$ux_1$ in~$N$.
        The tree~$T$ with an additional path $w_1up_zz$ is a base tree of~$N$.
        Note here that we could also have added a leaf to the arc~$uv$.
        \item $w_1x_1$ and~$w_2x_2$ are both linking arcs.
        We add a leaf~$z$ to the arc~$vx_2$ in~$N$.
        The tree~$T$ with an additional path~$w_1uvp_z z$ is a base tree of~$N$.
    \end{enumerate}
    So~$L_{\cTB}(N)=1$ in all four cases and thus the base case holds.

    Suppose now that the claim holds for all networks where its~$A^*_{\cTB}$ measure is at most~$k-1$, where~$k>1$.
    We show that the claim holds for a network~$N$ with~$A^*_{\cTB}(N) = k$.
    Let~$uv$ be the final element in the sequence~$A^*_{sol}$.
    Let~$N'$ denote the network obtained by applying the first~$k-1$ valid arc deletions in~$A^*_{sol}$.
    Then removing~$uv$ from~$N'$ yields a tree-based network.
    From the proof of the base case above, one can also add a leaf either to the outgoing arc of~$u$ which is not~$uv$, or a leaf to the outgoing arc of~$v$, to make~$N'$ tree-based.
    We wish to reflect this change to the original network~$N$.
    In the former case, let~$N''$ denote the network obtained by adding a leaf~$z$ to one of the outgoing arcs of~$u$; in the latter case, let~$N''$ denote the network obtained by adding a leaf~$z$ to the outgoing arc of~$v$. 
    Observe that applying the first~$k-1$ elements of~$A^*_{sol}$, with potentially changes in some of the arc endpoints as a result of adding~$p_z$, to~$N''$ gives a tree-based network.
    
    Then~$A^*_{\cTB}(N'')\le A^*_{\cTB}(N)-1 = k-1$, and thus by induction hypothesis, we may add at most~$k-1$ leaves to~$N''$ to make it tree-based.
    It follows that we may add at most~$k$ leaves to~$N$ to make it tree-based.
\end{proof}
In the proof above, exactly one leaf is added to `mimic' a valid arc deletion.
In practice, one leaf could replace multiple valid arc deletions.
We give an example network~$N$ in \cref{fig:example_Ltb_l_Astb}, for which~$L_{\cTB}(N)=1$ and~$A^*_{\cTB}(N)=2$.

\ej{We have not been able to find similar incomparability results for tree-child networks.}
Consider the non-tree-child network~$N$ consisting of a single crown on~$4$ arcs (together with root arcs and arcs leading to the two leaves from the reticulations).
As~$N$ contains~$2$ omnians, we must add at least two leaves (for more details, see \Cref{lem:L_TC=Omnian}).
On the other hand, removing any of the four reticulation arcs yields a network on a single reticulation, which must be tree-child.
So in this case,~$A^*_{\cTC}(N) = 1$ and $L_{\cTC}(N) = 2$.
We suspect that for tree-child networks, the valid arc deletion number is at most the leaf addition number, as we expect to find a sequence of valid arc deletions where at least one omnian is removed each time.
We give a conjecture on this in \Cref{sec:discussion}.

\subsection{\texorpdfstring{Comparing~$L$ and~$A$}{Comparing L and A}}\label{subsec:CompLA}

In this subsection, we show that the measures are incomparable for the classes~$\cOr$ and~$\cTB$ and comparable for~$\cTC$.

\begin{theorem}\label{thm:IncompLorAor}
    The measures $L_{\cOr}$ and $A_{\cOr}$ are incomparable.
\end{theorem}
\begin{proof}
As in the proof of \Cref{thm:IncompLorA*or}, we consider the networks~$N_1,N_2$ in~\Cref{fig:Lor_Aor_star_incomp}.
\ym{From the proof of \Cref{thm:IncompLorA*or}, we know that~$L_\cOr(N_1) = 1$ and~$L_\cOr(N_2) = 2$.}
\ym{It remains to show that~$A_{\cOr}(N_1)>1$ and that~$A_\cOr(N_2) = 1$.}
\paragraph{$A_\cOr(N_1)>1$:}
Note that all reticulation arcs that are not~$e_i$ are valid.
All leaves keep a reticulation parent in any of the graphs obtained by removing such arcs.
So at least one of the~$e_i$ arcs must be deleted.
As seen in the proof of \Cref{thm:IncompLorA*or}, deleting one of these arcs does not suffice.
So~$A_{\cOr}(N_1)>1$.

\paragraph{$A_\cOr(N_2) = 1$:}
By proof of \Cref{thm:IncompLorA*or},~$A^*_\cOr(N_2) = 1$. 
By~\Cref{obs:AleA^*},~$A_\cOr(N_2) \le A^*_\cOr(N_2) = 1$.
Finally, as the network~$N_2$ is non-orchard, we must have~$A_\cOr(N_2)\ge1$, meaning~$A_\cOr(N_2) = 1$. \qedhere
\end{proof}

\begin{theorem}\label{thm:IncompLTBATB}
    The measures $L_{\cTB}$ and $A_{\cTB}$ are incomparable.
\end{theorem}

\begin{figure}
    \centering
    \subfloat[]{\includegraphics[height=5cm]{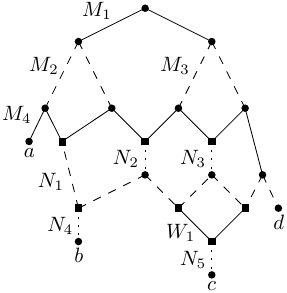}}
    \subfloat[]{{\includegraphics[height=5cm]{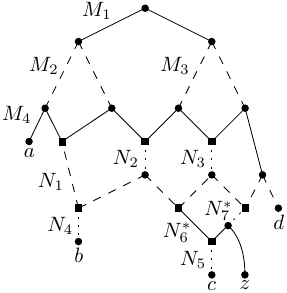}}
    \label{subfig:AddLeafLA_TB}}
    \hfill
    \subfloat[]{{\includegraphics[height=5cm]{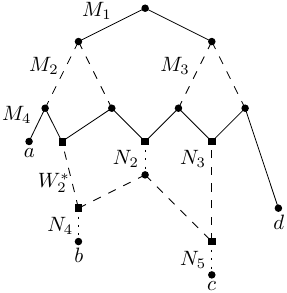}}
    \label{subfig:DeleteArcLA_TB1}}
    \subfloat[]{{\includegraphics[height=5cm]{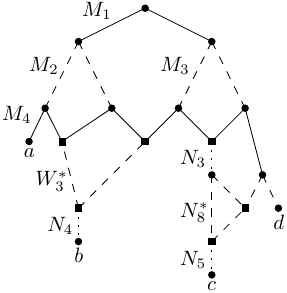}}
    \label{subfig:DeleteArcLA_TB2}}
    \caption{An example of a non-tree-based network~$N$ where~$L_\cTB(N) < A_\cTB(N)$.
    See \Cref{thm:IncompLTBATB} for a proof. (a) A non-tree-based network~$N$. (b) A network obtained by adding a leaf~$z$ to the W-fence~$W_1$ of~$N$. (c) A network obtained by deleting an arc in~$W_1$ of~$N$. (d) A network obtained by deleting an arc in~$N_1$ of~$N$.}
    \label{fig:counter_example_Ltb_l_Atb}
\end{figure}

\begin{figure}
    \centering
    \subfloat[]{\includegraphics[height=4cm]{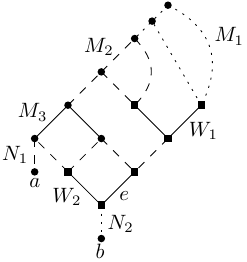}}
    \subfloat[]{\includegraphics[height=4cm]{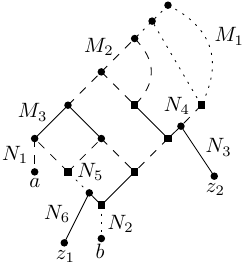}
    \label{subfig:Ltb_geq_Atb_fences_LeafAdd}}
    \hspace{0.03\textwidth}%
    \subfloat[]{\includegraphics[height=4cm]{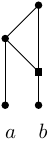}
    \label{subfig:Ltb_geq_Atb_fences_EdgeDelete}}
    \caption{An example of a non-tree-based network~$N$ where~$L_\cTB(N)>A_\cTB(N)$.
    See \Cref{thm:IncompLTBATB} for a proof. (a) A non-tree-based network~$N$. (b) $N$ with two leaves~$z_1,z_2$ added. (c) $N$ with the arc~$e$ removed.}
    \label{fig:counter_example_Ltb_g_Atb}
\end{figure}

\begin{proof}

\ym{We prove the lemma by using the two networks~$N, N'$ in~\Cref{fig:counter_example_Ltb_l_Atb,fig:counter_example_Ltb_g_Atb}, respectively (we refer to the network in \Cref{fig:counter_example_Ltb_g_Atb} as~$N'$ for convenience).} 
We show that~$L_{\cTB}(N) = 1$ and $A_{\cTB}(N) \ge 2$, and that~$L_\cTB(N') = 2$ and $A_\cTB(N') = 1$, thereby proving the claim. 
Note that in both \Cref{fig:counter_example_Ltb_l_Atb,fig:counter_example_Ltb_g_Atb}, the unique zig-zag decomposition for each network is indicated.
Our arguments will be based on the networks in the figures.

\paragraph{$L_\cTB(N) = 1$:}
The network~$N$ is non-tree-based, so~$L_\cTB(N)\ge 1$.
\Cref{subfig:AddLeafLA_TB} shows that one leaf suffices, as the zig-zag decomposition contains no W-fences upon adding~$z$.  
\paragraph{$A_\cTB(N) \ge 2$:}
\ym{Again, as~$N$ is non-tree-based,~$A_\cTB(N)\ge 1$.
Consider the W-fence~$W_1$.
If we remove a reticulation arc which has no endpoints in~$W_1$, then the resulting network will contain a W-fence.
This would mean that~$A_\cTB(N)\ge 2$.
On the other hand, if we remove a reticulation arc which has endpoints in~$W_1$, then we also end up with a W-fence in the resulting network. 
There are six possibilities here; we illustrate two in \Cref{subfig:DeleteArcLA_TB1,subfig:DeleteArcLA_TB2}.
The other 4 cases can be argued in the same way (by removing arcs).
This implies~$A_\cTB(N)\ge 2$.}
\paragraph{$L_\cTB(N') = 2$:}
\ym{Note that there are two W-fences~$W_1$ and~$W_2$ in~$N'$.
A leaf addition can remove at most one W-fence; it follows that~$L_\cTB(N') \ge 2$.
We see that~$L_\cTB(N') = 2$ from \Cref{subfig:Ltb_geq_Atb_fences_LeafAdd}. \footnote{Note that $L_\cTB(N)$ is equal to the number of W-fences. See \Cref{lem:L_TB=W-fences}.}
}

\paragraph{$A_\cTB(N') = 1$:}
\ym{Consider the network obtained by deleting the arc $e$.
All the 4 reticulations in the stack (3 above $e$) are removed, and the resulting network contains just~$1$ reticulation (\Cref{subfig:Ltb_geq_Atb_fences_EdgeDelete}).
Such a network must be} tree-based.
Therefore,~$A_{\cTB}(N')=1$.



\end{proof}

While the following lemma fits better in \Cref{subsec:LTC}, we include it here as we need it to prove \Cref{thm:CompLTCATC}.

\begin{lemma}\label{lem:L_TC=Omnian}
    Let~$N$ be a network. Then $L_{\cTC}(N)$ is equal to the number of omnians. 
    Moreover,~$N$ can be made tree-child by adding a leaf to exactly one outgoing arc of each omnian. 
\end{lemma}
\begin{proof}
	By definition, a network is tree-child if and only if it contains no 
	omnians. We show that every leaf addition can result in a network with one 
	omnian fewer than that of the original network.
	Let~$uv$ be an arc where~$u$ is an omnian. Add a leaf~$x$ to~$uv$. In the 
	resulting network,~$u$ has a child (the parent of~$x$) that is a tree 
	vertex, and it is no longer an omnian. The newly added tree vertex has a 
	leaf child~$x$; the parent-child combinations remain unchanged for the rest 
	of the network, so at most one omnian (in this case~$u$) can be removed per 
	leaf 
	addition. It follows that~$L_{\cTC}(N)$ is at least the number of omnians 
	in~$N$. By targeting arcs with omnian tails, we can remove at least one 
	omnian for every leaf addition, so that~$L_{\cTC}(N)$ is at most the number 
	of omnians in~$N$. Therefore,~$L_{\cTC}(N)$ is exactly the number of 
	omnians in~$N$.
\end{proof}

\begin{theorem}\label{thm:CompLTCATC}
    The measures~$A_\cTC$ and~$L_\cTC$ are comparable.
    In particular, for all networks~$N$, we have $A_{\cTC}(N) \leq L_{\cTC}(N)$.
\end{theorem}
\begin{proof}
    Let~$N$ be a network. By \Cref{thm:UniqueMaxZigZag}, it has a unique zig-zag decomposition.
    Noting that~$L_{\cTC}(N)$ is the number of omnians (\Cref{lem:L_TC=Omnian}), we wish to show that we can reduce the number of omnians by at least one in a network with a single arc deletion. 
    This shows that one can `emulate' leaf additions using arc deletions. 
    We first show that if~$N$ contains an M-fence of length at least 6, an N-fence of length at least 5, or crowns, then there exists an arc deletion that reduces the number of omnians by at least one.
    Once we have shown that, we may assume without loss of generality that our network~$N$ contains M-fences of length at most 4, N-fences of length 1, or W-fences. Subsequently, all omnians in~$N$ must be peaks of W-fences (i.e., all its children are contained in the same fence); we argue that arcs may be deleted from W-fences in a certain manner to ensure the resulting network contains fewer omnians than that of~$N$. Repeating this argument until the network contains no more omnians shows the required result.

    Suppose that~$N$ contains an M-fence~$(a_1,\ldots, a_k)$ of length~$k\ge 6$. Observe that $\tail(a_3)$ is an omnian.
    We delete the arc~$a_2$ and call the resulting network~$N'$. Upon cleaning up, we suppress vertices $\tail(a_2)$ and $\head(a_2)$.
    Since~$\head(a_1)$ is necessarily a tree vertex by definition of M-fences, the parent of $\head(a_1)$ in~$N'$, which corresponds to the parent of~$\tail(a_2)$ in~$N$, cannot be an omnian in~$N'$.
    Observe that $\head(a_2)$ is an omnian in~$N$ if and only if $\tail(a_3)$ is an omnian in~$N'$. 
    Therefore, if~$\head(a_2)$ was an omnian in~$N$, then it ceases to be an omnian in~$N'$ since it is suppressed. On the other hand, if~$\head(a_2)$ was not an omnian in~$N$, then~$\tail(a_3)$ is no longer an omnian in~$N'$. No other new omnians are created. This means in either situation, the number of omnians in~$N'$ is exactly one fewer than that in~$N$.


    \ej{Suppose now that~$N$ contains an N-fence~$(a_1,\ldots,a_k)$ of length~$k\ge 5$. If the arc~$a_2$ is valid, then the same argument as above applies by deleting~$a_2$. Therefore, we consider the case when~$a_2$ is invalid.}

Since the endpoints of~$a_k$ are both tree vertices, it follows from \Cref{lem:ValidEdgeCharacterization} that $k=5$. 
In this case, we consider the local structure containing $(a_1,\ldots,a_5)$.
We are in either Case 2 or 4 of \Cref{lem:ValidEdgeCharacterization}.
If we are in Case 4, then we may simply remove the arc~$a_3$, which creates a parallel arc that is subsequently removed.
This reduces the number of omnians by~$2$ using one arc deletion (the omnians are $\tail(a_1)$ and $\tail(a_2)$).
So we are in Case 2.
Observe that although $a_2$ is invalid, there exists another arc in this structure (namely $a_3$) whose deletion is valid.
Let $N'$ be obtained from $N$ by deleting $a_3$ and cleaning up. 
\ym{If~$\head(a_3)$ was an omnian in~$N$, then~$\tail(a_4) = \tail(a_5)$ is now an omnian in~$N'$.
As~$\head(a_3)$ and $\tail(a_3)$ become suppressed in~$N'$, the number of omnians has decreased by one; indeed, all other vertices retain their status as omnian or non-omnian.
On the other hand, if~$\head(a_3)$ was not an omnian in~$N$, then no new omnians are created, and we reduce the number of omnians by one (by suppressing $\tail(a_2)$).
}

\ym{The case when~$N$ contains a crown follows a similar argument as for the M-fence of length at least~$6$, as all arcs in the crown would be valid.}
\medskip

\ym{So suppose that~$N$ contains M-fences of length at most 4, N-fences of length 1, or W-fences. 
All omnians must be peaks of W-fences.
If there is a W-fence~$(a_1,\ldots, a_k)$ of length~$k\ge 4$, then we remove~$a_2$.
Then the parent of~$\tail(a_2)$ becomes an omnian and~$\tail(a_2)$ ceases to be an omnian as it is suppressed. 
If the child of~$\head(a_2)$ is a reticulation, then~$\head(a_2)$ ceases to be an omnian as it is suppressed, and all other vertices retain their omnian and non-omnian status. 
On the other hand, if the child of~$\head(a_2)$ is not a reticulation, then $\tail(a_2)$ becomes a non-omnian.
In both cases, the number of omnians decreases by one by an arc deletion.}

\ym{So finally suppose that all W-fences are of length~$2$. Given a W-fence, follow the `stack' of reticulations upwards until we reach a tree vertex.
Since all M-fences are of length at most 4 and since all N-fences are of length at most~$1$, all of the reticulation arcs in such a path must all be contained in distinct W-fences.
We remove the highest reticulation arc~$a$ in such a path (which must necessarily be contained in an M-fence). 
This creates no omnians, while reducing by one the number of omnians, by reducing one from the W-fence immediately below~$a$. This completes the proof. \qedhere}


\end{proof}

\section{\texorpdfstring{$L$: Leaf Addition}{L: Leaf Addition}}
\label{sec:LeafProximity}

Recall that~$L_{\cC}(N)$ denotes the minimum number of leaf additions required to make the network~$N$ a member of~$\cC$ (\Cref{def:leaf_add}). 
We now give complexity results and bounds on the leaf addition proximity measure for tree-child, orchard, and tree-based networks. We show that for tree-child and tree-based networks, the measure can be computed in polynomial time (\Cref{thm:L_TCisPoly,thm:L_TB=polynomial}), but for orchard networks the problem is NP-hard to compute (\Cref{thm:L_Or=Hard}). This hardness proof is given in \Cref{sec:Hardness}.

\subsection{Tree-Child Networks}\label{subsec:LTC}
We show in \Cref{thm:L_TCisPoly} that this number can be computed in polynomial time. 
In \Cref{thm:L_TCBound} we give an upper bound based on the number of reticulations in the network.

\begin{theorem}\label{thm:L_TCisPoly}
    Let~$N$ be a network. Then~$L_{\cTC}(N)$ can be computed in~$O(|N|)$ time.
\end{theorem}
\begin{proof}
    We first show that the number of omnians of~$N$ can be computed in~$O(|N|)$ 
    time, by checking, for each vertex, the indegrees of its children. A vertex 
    is an omnian if and only if all of its children are of indegree-2. Since 
    the degree of every vertex is at most 3, each search within the for loop (over all vertices)
    takes constant time. The for loop iterates over the vertex set of size~$O(|N|)$. 
    By \Cref{lem:L_TC=Omnian}, since~$L_{\cTC}(N)$ is the number of omnians in~$N$, we can compute~$L_{\cTC}(N)$ in~$O(|N|)$ time.
\end{proof}

\begin{theorem}\label{thm:L_TCBound}
   Let $N$ be a network with~$r$ reticulations. If~$r\le 1$, then~$L_{\cTC}(N) = 0$. Else, $L_{\cTC}(N) \leq \floor{\frac{3}{2}r - \frac{1}{2}}$. \ym{When~$r>1$, this bound is tight.}
\end{theorem}
\begin{proof}
    If~$r\le 1$, then~$N$ is either a tree or a network with a single reticulation.
    In the latter case, every tree vertex has at least one non-reticulation child, and the single reticulation must have a non-reticulation child, so it must be tree-child.
    So suppose that~$N$ contains~$r\ge2$ reticulations,~$t$ omnians that are tree vertices, and~$s$ omnians that are reticulations. 
    Every tree vertex omnian is the tail of two reticulation arcs; every reticulation omnian is the tail of one reticulation arc. 
    Observing that every arc has exactly one tail, we must have that
    $s+2t \le 2r$.
    Rearranging, and using the fact that lowest reticulations cannot be omnians, so~$s\le r-1$, gives 
    \[L_{\cTC}(N) = \text{Number of omnians in~$N$} = s+t \le r + \frac{s}{2} \le \frac{3r-1}{2}.\]
    Since~$L_\cTC(N)$ is an integer, we may take the floor function.

We now show that for~$r>0$, the bound is tight. 
For~$r=2$, consider the network with two reticulations that share the same two parents (i.e., a crown on $4$ arcs). The parents are both omnians. For the~$r\ge 3$ case, 
we construct the following network on~$r$ reticulations (see \Cref{fig:TightL_TC}).
Let~$v_1\ldots v_r$ be a path of reticulations. 
Observe that we maximize the number of omnians if all but one (the lowest) reticulation is an omnian. 
Since exactly one reticulation arc is used per reticulation omnian in the above, there are $r+1$ reticulation arcs that can be used to construct tree vertex omnians. Each tree vertex omnian requires two reticulation arcs, meaning that we can have a maximum of $k= \floor{\frac{r+1}{2}}$ tree vertex omnians. 
We can realize this maximum by adding tree vertices~$u_1,\ldots, u_k$, and adding arcs $u_1v_1, u_1v_2, u_2v_1, u_2v_3$, and arcs $u_iv_{2i-2}, u_iv_{2i-1}$ for $3\le i\le k$.
Connect all vertices whose degree requirements are not yet met arbitrarily, taking care that the network remains acyclic (by adding a root, tree vertices, and labelled leaves).  
Such a network has~$r-1+k = r-1+\floor{\frac{r+1}{2}} = \floor{\frac{3}{2}r - \frac{1}{2}}$ omnians.
\end{proof}

\begin{figure}
    \centering
    \includegraphics[height=5cm]{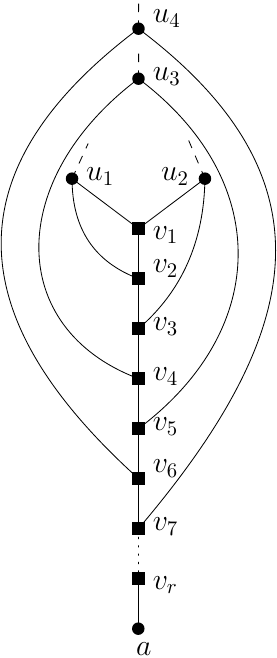}
    \caption{A network with~$r$ reticulations which contains~$\floor{\frac{3}{2}r - \frac{1}{2}}$ omnians. See the final paragraph of \Cref{subsec:LTC} for the explanatory text.}
    \label{fig:TightL_TC}
\end{figure}

\subsection{Orchard Networks}
Interestingly, computing~$L_\cOr(N)$ proves to be a difficult problem, although the leaf addition proximity measure is easy to compute for its neighbouring network classes. We prove the following result in~\Cref{sec:Hardness}.
\begin{thmn}[\ref{thm:L_Or=Hard}]
	Let~$N$ be a network. Computing $L_\cOr(N)$ is NP-hard. 
\end{thmn}
We also include the following theorem which states that when considering leaf addition proximity measures for orchard networks, it suffices to consider leaf additions to reticulation arcs. We shall henceforth assume that all leaf additions are on reticulation arcs.
\begin{theorem}[Theorem 4.1 of \cite{susanna2022making}]\label{thm:OnlyAddToRetArcs}
    A network~$N$ is orchard if and only if the network obtained by adding a leaf to a tree arc of~$N$ is orchard.
\end{theorem}

We now provide a sharp upper bound for~$L_{\cOr}(N)$.
We call a reticulation \emph{highest} if it has no reticulation ancestors.

\begin{lemma}
\label{lem:RetHasLeafSibling}
    Let~$N$ be a network. Suppose there is a highest reticulation~$r$ such that all other reticulations have a leaf sibling. Then~$N$ is orchard.
\end{lemma}
\begin{proof}
    We prove the lemma by induction on the number of reticulations~$k$. For the base case, observe that a network with one reticulation is tree-child since it has no omnians. A tree-child network is orchard~\cite{janssen2021cherry}, and so this network must be orchard.

    Suppose now that we have proven the lemma for all networks with fewer than~$k$ reticulations, where~$k>1$. 
    Let~$N$ be a network with reticulation set~$R$ where~$|R| = k$, and suppose there exists a highest reticulation~$r$ in~$N$ such that all other reticulations have a leaf sibling.
    Let~$r$ denote the highest reticulation as specified in the statement of the lemma.
    Choose a lowest reticulation~$r'\in R\setminus \{r\}$. By assumption,~$r'$ has a leaf sibling~$c$.
    Every vertex below~$r'$ must be tree vertices and leaves. Reduce cherries until the child~$x$ of~$r'$ is a leaf. 
    Then~$(x,c)$ is a reticulated cherry; the network~$N'$ obtained by reducing this reticulated cherry has $k-1$ reticulations and has a highest reticulation~$r$ such that all other reticulations have a leaf sibling.
    By induction hypothesis,~$N'$ must be orchard. Since a sequence of cherry reductions can be applied to~$N$ to obtain~$N'$, the network~$N$ must also be orchard.
\end{proof}

\begin{theorem}
\label{thm:LORUpperBound}
    Let $N$ be a network, and let~$r$ denote the number of reticulations. Then $L_\cOr(N)=0$ if~$N$ is a tree, and otherwise, $L_{\cOr}(N) \leq r - 1$, where the bound is sharp.
\end{theorem}
\begin{proof}
    If~$N$ is a tree, then it is orchard, and so~$L_\cOr(N) = 0$.
    So suppose~$r>0$. Let $r$ be a highest reticulation of~$N$, and for every other reticulation, arbitrarily choose one incoming reticulation arc.
    Add a leaf to each of these reticulation arcs. By \Cref{lem:RetHasLeafSibling}, the resulting network must be orchard.
    We have added a leaf for all but one reticulation in~$N$. It follows that~$L_\cOr(N) \le r - 1$.
    
    To show that the bound is sharp, consider a path~$v_1\ldots v_r$ of reticulations (see the network in \Cref{fig:Lor_bound} for an example when~$r=5$).
    Let~$\rho$ denote the root, let~$u_1\ldots u_r$ be a path of tree vertices, and let~$a,b$ be two leaves.
    Add the arcs~$\rho u_1, \rho v_1$ and also the arcs~$u_i v_{i+1}$ for~$i\in [r-1]$. 
    Finally add arcs~$u_r v_1, u_r a,$ and~$v_rb$.
    Observe that the only way to reduce the graph using cherry-picking sequences is to add leaves to the arcs $u_i v_{i+1}$ for~$i\in [r-1]$.
\end{proof}

\begin{figure}
    \centering
    \includegraphics[width=0.22\columnwidth]{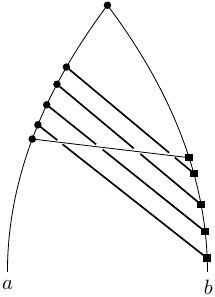}
    \caption{A network~$N$ on two leaves~$\{a,b\}$ with~$r = 5$ reticulations. Observe that $L_{\cOr}(N) = r - 1 = 4$, since the highest reticulation cannot be reduced by cherry picking unless the reticulations below it are first reduced. For each non-highest reticulation, we must add a leaf to one of its incoming arcs to reduce it, which leads to $L_{\cOr}(N) = 4$. Note that this construction can be extended for any number of reticulations. 
    }
    \label{fig:Lor_bound}
\end{figure}

\subsection{Tree-Based Networks}

Similarly to tree-child networks where $L_\cTC$ is equivalent to \textit{omnians}, tree-based networks are also characterized by a forbidden substructure, namely, the W-fences (\Cref{lem:TB=NoW}). We first reiterate this equivalence in \Cref{lem:L_TB=W-fences} and its complexity in \Cref{thm:L_TB=polynomial}, and subsequently derive its bound in \Cref{thm:L_TBBound}.

It has been shown already that~$L_{\cTB}(N)$ can be computed in~$O(|N|^{3/2})$ time 
where~$|N|$ is the number of vertices in~$N$~\cite{francis2018new}. 
This was shown to be solvable in $O(|N|)$ time by adding a leaf to every W-fence~\cite[Corollary 5.4]{hayamizu2021structure}.
We include the proof here for completeness.

\begin{lemma} \label{lem:L_TB=W-fences}
    Let $N$ be a network. Then $L_{\cTB}(N)$ is equal to the number of W-fences. Moreover, $N$ can be made tree-based by adding a leaf to any arc in each W-fence in $N$.
\end{lemma}
\begin{proof}
    By \Cref{lem:TB=NoW}, a network is tree-based if and only if it 
	contains no W-fences. We show that every leaf addition can result in a 
	network with one W-fence fewer than that of the original network.
	Suppose that~$N$ contains at least one 
	W-fence. 
	Otherwise we may conclude that the network is tree-based by 
	\Cref{lem:TB=NoW}. 
	Let~$(a_1,a_2,\ldots, a_k)$ be a W-fence in~$N$ where $a_i = u_iv_i$ 
	for~$i\in[k]$, and add a leaf~$x$ to~$a_1$; let~$p_x$ be the tree vertex 
	parent of~$x$. In the resulting network, the arcs 
	in~$\{u_1p_x,p_xv_1,p_xx, a_2, a_3,a_4,\ldots, a_k\}$ are decomposed into 
	their unique maximal zig-zag trails (\Cref{thm:UniqueMaxZigZag}) as two 
	N-fences~$(u_1p_x)$ and~$(a_k,a_{k-1},\ldots,a_3,a_2,p_xv_1,p_xx)$. All 
	other arcs remain in the same maximal zig-zag trails as in~$N$. 
	Therefore the number of W-fences has gone down by exactly one. This can be 
	repeated for all W-fences in the network; it follows that~$L_{\cTB}(N)$ is 
	the number of W-fences in~$N$.

    A quick check shows that adding a leaf to any arc in the W-fence decomposes the W-fence into two N-fences.
\end{proof}

\begin{theorem}\label{thm:L_TB=polynomial}
	Let~$N$ be a network. Then $L_\cTB(N)$ can be computed in~$O(|N|)$ time. 
\end{theorem}
\begin{proof}
   Finding the maximal zig-zag decomposition takes $O(|N|)$ time (Proposition 5.1 of \cite{hayamizu2021structure}). Counting the number of W-fences in the decomposition gives~$L_\cTB(N)$ by \Cref{lem:L_TB=W-fences}.
\end{proof}

\begin{figure}[ht]
    \centering
    \includegraphics{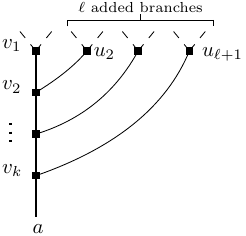}
    \caption{\ym{Construction used in proof of \Cref{thm:L_TBBound} to show tightness of bound.
    A network~$N$ with an odd number $r = k+\ell$ of reticulations, where
    $L_{\cTB}(N) = \left\lfloor \frac{r-1}{2} \right\rfloor$.}}
    \label{fig:Ltb_bound}
\end{figure}

We will now derive a sharp upper bound to $L_\cTB$, which is based on bounding the number of possible W-fences in a network.
\begin{theorem}\label{thm:L_TBBound}
    Let $N$ be a network with~$r\ge1$. Then $L_{\cTB}(N) \leq \left\lfloor \frac{r-1}{2} \right\rfloor$.
    \ym{This bound is tight.}
\end{theorem}
\begin{proof}
    By \Cref{lem:L_TB=W-fences}, we have that $L_{\cTB}(N)$ is equal to the number of W-fences in $N$. Each W-fence starts and ends with a reticulation arc whose tails are also reticulations. Let us refer to such tails as \emph{W-fence endpoints}.
    Observe that W-fence endpoints are distinct as every reticulation has one outgoing arc, and since every zig-zag decomposition of any network is unique by \Cref{thm:UniqueMaxZigZag}.
    Note that each~W-fence has two W-fence endpoints, and that there exists at least one reticulation that cannot be a W-fence endpoint, namely lowest reticulations in~$N$.
    Since~$L_{\cTB}(N)$ is equal to the number of W-fences in~$N$ by \Cref{lem:L_TB=W-fences}, it follows then that
    \[L_{\cTB}(N) = \text{Number of W-fences}\le \frac{r-1}{2}.\]

\ym{We show that this bound is tight. See \Cref{fig:Ltb_bound} for an illustration.}
We construct a network~$N$ with~$r$ reticulations with~$L_{\cTB}(N) = \floor{\frac{r-1}{2}}$ as follows. Let~$k=\ceil{\frac{r+1}{2}}$. Let~$v_1\ldots v_k$ be a path of reticulations. Let $\ell := r - k$, and let~$u_2,\ldots, u_{\ell+1}$ be the rest of the reticulation vertices. Add arcs~$u_iv_i$ for~$i = 2,\ldots, \ell+1$, and connect all vertices whose degree requirements are not yet met arbitrarily, taking care that the network remains acyclic (by adding a root, tree vertices, and labelled leaves). Such a network contains~$\floor{\frac{r-1}{2}}$ W-fences.
\end{proof}

\section{\texorpdfstring{$A^*$: Valid Arc Deletion}{A*: Valid Arc Deletion}}\label{sec:A*Deletion}
As defined earlier,~$A^*_{\cC}(N)$ is the minimum number of valid arc deletions to make network~$N$ a member of class~$\cC$ (\Cref{def:val_arc_del}). 
\ym{Recall that we call a reticulation arc~$e$ \emph{valid} if deleting~$e$ and cleaning up results in a network that has exactly three fewer arcs than that of the original network.}

Compared to leaf addition, valid arc deletions are different, in that the order in which the arcs are deleted, matters.
For example, consider a reticulation~$r_1$ that is a parent of another reticulation~$r_2$. Let~$u,v$ denote the two parents of~$r_1$. Suppose~$ur_1$ is a valid arc; upon deleting this, we note that $vr_2$ is now a valid arc (assume this is the case). Note that $vr_2$ was not a valid arc, let alone an arc, in the original graph. Therefore, when considering valid arc deletions, we consider a sequence~$A = (a_1,\ldots, a_k)$ of valid arcs, such that $a_i$ is a valid arc in the network obtained by removing $(a_1,\ldots, a_{i-1})$ sequentially from the original network.


\begin{figure}
    \centering
    \noindent\makebox[\textwidth][c]{%
    \subfloat[]{\includegraphics[width=0.23\textwidth]{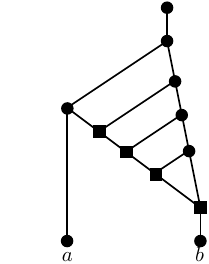}}
    \hfill
    \subfloat[]{\includegraphics[width=0.23\textwidth]{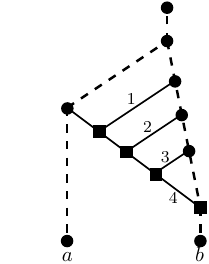}}
    \hfill
    \subfloat[]{\includegraphics[width=0.23\textwidth]{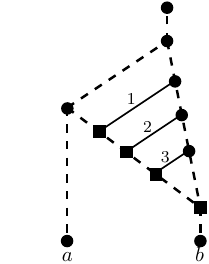}}
    \hfill
    \subfloat[]{\includegraphics[width=0.23\textwidth]{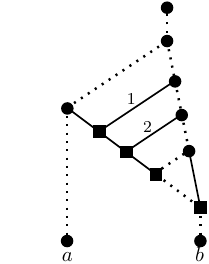}}
    }
    \caption{Example of valid arc deletion proximity measure for all studied network classes. The \ym{bold} numbered arcs give the sequence of arc deletions per network class. (a) The original network $N$. (b) The sequence of deletions to construct a tree, where the dashed arcs represent the tree embedding ($A^*_{\cT}(N) = 4$). (c) The sequence to construct both a tree-child and an orchard network, where the dashed arcs represent the final network ($A^*_{\cTC}(N) = A^*_{\cOr}(N) = 3$). (d) The sequence for obtaining a tree-based network, where the dotted arcs represent the base tree of the final network ($A^*_{\cTB}(N) =2$).}
    \label{fig:Ast_example}
\end{figure}

\ej{Just like the leaf addition measure, we will find that computing~$A^*_\cOr(N)$ is NP-hard. We prove the following result in~\Cref{sec:Hardness}, where we use the same reduction as for the hardness proof of computing $L_{\cOr}(N)$.}

\begin{thmn}[\ref{thm:A*_Or=Hard}]
	Let~$N$ be a network. Computing~$A^*_\cOr(N)$ is NP-hard. 
\end{thmn}
\ej{We continue by introducing a subgraph type called \textit{OSOS}. This structure is used to characterize the feasibility, or finiteness, of the $A^*$ measure.}

\subsection{OSOS-subgraph}
In the following series of results, we examine the nature of so-called `one source one sink subgraphs', abbreviated as \emph{OSOS-subgraphs}; see \Cref{fig:osos} for an illustration and the following definition.
\begin{definition}[OSOS-subgraph]
    Let $N$ be a network \ym{with potentially vertices of indegree and outdegree $1$}, and let~$u$ and~$v$ be vertices in~$N$ such that~$u$ is above~$v$,~$u$ is a tree vertex, and~$v$ is a reticulation. We say that~$uv$ induces an OSOS-subgraph in~$N$ if 
    \begin{itemize}
        \item every path from the root to $v$ contains~$u$, and
        \item every path from~$u$ to a leaf contains~$v$, and
    \end{itemize}
    In addition, if $u$ is the lowest vertex that satisfies the above two conditions (or equivalently, $v$ is the highest vertex that satisfies the above two conditions), we say that $uv$ induces a \emph{minimal OSOS}.
    Note that a network contains an OSOS-subgraph if and only if it contains a minimal OSOS.
    For this reason, we use OSOS to refer to a minimal OSOS from this point on.
    Whenever~$uv$ induces an OSOS-subgraph, we shall talk about this OSOS-subgraph as the graph induced by all vertices that are descendants of~$u$ and are ancestors of~$v$.
\end{definition}
\begin{observation}\label{obs:OSOSLaminar}
    \ym{Let~$w$ be a vertex in an OSOS induced by~$uv$. Then every path from the root to~$w$ contains~$u$.}
\end{observation}

In compiler theory, specifically in the study of control-flow graphs, one says that a node~$u$ \emph{dominates} another node~$v\ne u$ if every path from the root to~$v$ contains~$u$ \cite{lengauer1979fast}.
We call a vertex a \emph{dominator} if it dominates another vertex.
If~$u$ dominates~$v$ and if every other dominator of~$v$ dominates~$u$, then we say that~$u$ \emph{immediately dominates}~$v$. Note that such nodes are unique. 
We call a vertex an \emph{immediate dominator} if it immediately dominates another vertex.
Note that OSOS-subgraphs must necessarily be induced by tree vertices which immediately dominate reticulations.

\begin{figure}
    \centering
    \subfloat[]{\includegraphics{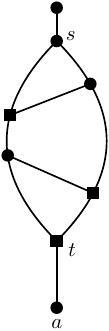} \label{subfig:osos_tb}}
    \hspace{10mm}
    \subfloat[]{\includegraphics{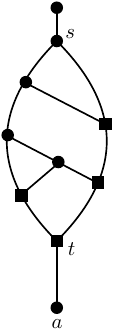}}
    \caption{Two examples of networks with OSOS-subgraphs (between nodes $s, t$).}
    \label{fig:osos}
\end{figure}

We say that a network is \emph{OSOS-free} if no pair of vertices induces an OSOS-subgraph. A similar type of graph, called a \emph{two-terminal series parallel (TTSP) graph} is investigated in the field of algorithm design, to speed-up certain computations on graphs in linear time~\cite{eppstein1992parallel}. Though similar, the two definitions are different. TTSP graphs contains either a single arc between two vertices, a parallel arc, or vertices of indegree-1 and outdegree-1. This means that TTSP graphs cannot be networks, and therefore they cannot be OSOS-subgraphs, except for the single arc case. TTSP graphs can be thought of as specific `one source one sink' subgraphs, if we were to generalize the definition of OSOS-subgraphs. We start by investigating the structure of OSOS-subgraphs.

\begin{lemma}\label{lem:OSOShas3r-1arcs}
    Let~$N$ be a network where two vertices~$u$ and~$v$ induce an OSOS-subgraph $S$.
    If~$S$ contains~$r$ reticulations, then~$S$ has~$3r-1$ arcs.
\end{lemma}
\begin{proof}
    Observe that an OSOS-subgraph~$S$ induced by~$u$ and~$v$ contains only tree vertices and reticulations, except for~$u$ and~$v$ which have~indegree-0, outdegree-2, and indegree-2, outdegree-0, respectively. 
    Since there is exactly one source and one sink, the number of tree vertices must be equal to the number of reticulations.
    So in total,~$S$ contains~$2r$ vertices where~$r$ is the number of reticulations in~$S$.
    To find the number of arcs in~$S$, we use the handshaking lemma on the fact that every vertex in~$S$ is of degree~$3$ while~$u$ and~$v$ have degree~$2$.
    \[\text{Number of arcs in~$S$} = \displaystyle\frac{2+2+3(2r-2)}{2} = 3r-1.\]
\end{proof}

\begin{lemma}\label{lem:OSOS-freeArcDeletion}
    Let~$N$ be an OSOS-free network. \ym{Then either~$N$ is a tree,} or there exists a valid arc such that 
    deleting it from $N$ results in a network that is OSOS-free.
\end{lemma}
\begin{proof}
	We prove the lemma by induction on the reticulation number of~$N$. For the 
	base case, consider the cases when~$N$ is a tree and when~$N$ has just one 
	reticulation. If~$N$ is a tree, then it contains no reticulation arcs, and 
	therefore no valid arcs. The claim is vacuously true. When~$N$ has just one 
	reticulation, then both reticulation arcs are valid. Deleting either arcs
	results in a tree, which is, by definition, OSOS-free. Therefore the base 
	cases have been shown.
	
	For the inductive step, suppose we have shown the claim to be true for 
	all OSOS-free networks with reticulation number fewer than~$k$, where~$k\ge 
	2$. Let~$N$ be an OSOS-free network on $k$ reticulations. Consider a 
	highest reticulation~$r$ with arcs~$ur$ and~$vr$.
    \ym{If either of the deletions~$ur$ or~$vr$ results in an OSOS-free network, then we are done.
    So suppose that the two networks obtained by deleting~$ur$ and by deleting~$vr$, respectively, both contain an OSOS.}
    Let~$N'$ denote the network obtained by deleting~$ur$, 
	and not suppressing~$u$ and~$r$. 
    \ym{Since~$u$ and~$r$ are degree-2 vertices in~$N'$, we have that~$N'$ contains an OSOS-subgraph if and only if the network obtained by cleaning up~$N'$ contains an OSOS}.
    Note that an OSOS-subgraph of~$N'$ must contain the vertex~$u$ or~$r$. 
	It cannot contain neither, and it cannot contain both vertices~$u$ and~$r$, as otherwise~$N$ must also have contained an OSOS. 
    A similar statement holds for 
	the network~$N''$ obtained by deleting~$vr$ and not suppressing~$v$ 
	and~$r$. 
	
	We first show that $u$ is contained in an OSOS-subgraph in~$N'$ or that $v$ 
	is contained in an OSOS-subgraph in~$N''$. Suppose for a contradiction that 
	this was not true, that~$u$ and~$v$ are not contained in an OSOS-subgraph 
	in~$N'$ and~$N''$, respectively. This means that~$v,r$ are contained in an 
	OSOS-subgraph induced by, say,~$p_vq_v$ in~$N'$, and the vertices~$u,r$ 
	are contained in an OSOS-subgraph induced by, say,~$p_uq_u$ in~$N''$. 
    \ym{If~$q_u$ and~$q_v$ are incomparable, then neither $p_vq_v$ nor~$p_u q_u$ can induce an OSOS-subgraph in~$N'$ and~$N''$, respectively.
    So they must be comparable.}
    Without loss of generality, 
	suppose~$q_u$ is above~$q_v$. By definition of OSOS-subgraphs,~$p_v$ must 
	be above~$q_u$. 
    \ym{Since~$p_v q_v$ induces an OSOS-subgraph in~$N''$, all paths from the root to~$q_u$ must contain~$p_v$ by \Cref{obs:OSOSLaminar}.
    So~$p_v$ must be above~$p_u$.}
    But this gives a contradiction as~$p_vq_v$ would induce 
	an OSOS-subgraph in~$N$. 
    So~$u$ is contained in an OSOS-subgraph in~$N'$ or 
	$v$ is contained in an OSOS-subgraph in~$N''$. Without loss of generality, 
	suppose that~$u$ is contained in an OSOS-subgraph in~$N'$.
	
	Let us denote this OSOS-subgraph as one being induced by~$p_uq_u$. 
    Consider the network~$S'$ obtained by adding 
	leaves~$x,y$ and arcs~$ux,q_uy$ to this OSOS-subgraph.
    \ym{Since $S'$ is obtained from~$N$ by deleting arcs and attaching leaves that are not involved in any OSOS structure, any OSOS-subgraph in~$S'$ corresponds to an OSOS-subgraph in~$N$. Thus~$S'$ is OSOS-free.} 
	Observe that~$S'$ has at most~$k-1$ reticulations, as it does not contain 
	the reticulation vertex~$r$. By induction hypothesis, there exists a valid 
	arc~$e$ in~$S'$ such that deleting it results in a network that is 
	OSOS-free. Clearly,~$e$ is also a valid reticulation arc in~$N$, since validity is locally contained. We claim 
	that deleting~$e$ results in a network that is OSOS-free. But this is 
	immediate. One can view~$S'$ as an `one source two sink subgraph' of~$N$; 
	deleting~$e$ will only affect the network locally, in particular, only 
	in~$S'$. If deleting~$e$ induces an OSOS-subgraph in~$N$, then this would 
	induce an OSOS-subgraph in~$S'$, which would be a contradiction. Therefore, 
	$e$ is a valid arc in~$N$ such that deleting it results in a network that 
	is OSOS-free.
\end{proof}

\subsection{\texorpdfstring{Feasibility of $A^*$}{Feasibility of A*}}
In the following, we use the OSOS characterization to give results for the feasibility, or finiteness, of $A^*$ for trees, tree-child and orchard networks. As the proofs of each of the network classes are intertwined, this section is not separated by network class. We start with the following lemma:

\begin{lemma}\label{lem:OrchHasFiniteA*T}
	Let~$N$ be an orchard network. Then $A^*_{\cT}(N)$ is finite.
\end{lemma}
\begin{proof}
	We prove the claim by induction on~$r$, the number of reticulations in~$N$. 
	When~$r=0$, the network~$N$ is a tree; when $r=1$, it is not possible to 
	have an invalid arc. So the base case is established. In both 
	cases,~$A^*_{\cT}(N)$ are finite. Suppose now 
	that~$r\ge 2$, and that we have proved the claim for all networks with 
	number of reticulations fewer than~$r$.
	
	Since~$N$ is orchard, we can find an acyclic cherry cover~$P$ by 
	\Cref{thm:AcyclicCherryCover}. Because the cover is acyclic, there must 
	exist a reticulated cherry shape~$\{uv, pu, pq\}$ that is lowest. We claim 
	that the middle reticulation arc~$pu$ is valid. As a lowest reticulated 
	cherry shape, $p$, and thereby other vertices of the shape, can be above 
	reticulations other than~$r$. This means that all vertices below~$p$ 
	excluding $u$ must be tree vertices or leaves. Furthermore, $p$ is a tree 
	vertex. By \Cref{lem:ValidEdgeCharacterization}, the arc~$pu$ must be valid.
	
	Let~$N'$ denote the network obtained by deleting~$pu$. We now claim 
	that~$N'$ is orchard. By Theorem 1 of \cite{janssen2021cherry}, the order 
	in which cherry reductions take place does not matter. So let~$ScS'$ be a 
	cherry-picking sequence for~$N$ where~$S$ reduces the subtree rooted at $v$ 
	and the subtree rooted at~$q$, and where~$c$ removes the arc~$pu$ as part 
	of a reticulated cherry reduction. It is easy to see that~$SS'$ is a 
	cherry-picking sequence for~$N'$. Therefore, $N'$ is an orchard network.
	
	The network~$N'$ is orchard and contains~$r-1$ reticulations. By induction 
	hypothesis, $A^*_{\cT}(N')$ is finite. Then we have that $A^*_{\cT}(N) = 
	A^*_{\cT}(N')+1$, which must also be finite.
\end{proof}

\begin{lemma} \label{lem:At_Aor_finite}
    Let~$N$ be a network. $A^*_{\cT}(N)$ is finite if and only if~$A^*_{\cOr}(N)$ is finite.
\end{lemma}
\begin{proof}
	Since a tree is orchard, we have that~$A^*_{\cOr}(N) \le A^*_{\cT}(N)$. 
	This establishes one direction of the proof. So suppose 
	that~$A^*_{\cOr}(N)$ is finite, in other words, that there exists a 
	sequence of $A^*_{\cOr}(N)$ valid arc deletions such that the resulting 
	network, say~$N'$, is orchard. By \Cref{lem:OrchHasFiniteA*T}, we have 
	that~$A^*_{\cT}(N')$ is finite. We have
	\[A^*_{\cT}(N) = A^*_{\cOr}(N) + A^*_{\cT}(N'),\]
	where the equality follows since the number of valid arc deletions is exactly the number of reticulations in~$N$.
    As both summands on the right-hand side are finite, we must have that $A^*_{\cT}(N)$ is 
	finite.
\end{proof}

Note that if~$A^*_{\cT}(N)$ is finite, it is equal to the number of 
reticulations of~$N$. 

\begin{theorem}\label{thm:FiniteLTiffOSOS-free}
    Let~$N$ be a network. Then~$A^*_{\cT}(N)$ is finite if and only if~$N$ is OSOS-free.
\end{theorem}
\begin{proof}
    Suppose first that~$N$ contains an OSOS-subgraph~$S$ with~$r$ reticulations.
    By \Cref{lem:OSOShas3r-1arcs}, $S$ contains~$3r-1$ arcs. After deleting~$r-1$ valid arcs (if such a sequence of valid arcs exist) from~$S$, since each valid arc deletion reduces the number of arcs by~$3$, the resulting OSOS-subgraph contains~$3r-1 - 3(r-1) = 2$ arcs. 
    An OSOS-subgraph with only two arcs must contain parallel arcs, which implies that after a sequence of $r-2$ arc deletions from an OSOS-subgraph (should such a sequence exist), we obtain an OSOS-subgraph that contains only invalid arcs.
    Note here that we write~$r-2$ since an OSOS-subgraph with~$2$ reticulations contain only invalid arcs.
    It follows then that~$A^*_{\cT}(N)$ cannot be finite.

    Suppose now that~$N$ is OSOS-free. By \Cref{lem:OSOS-freeArcDeletion}, there exists a valid arc such that deleting it results in a network that is OSOS-free.
    Repeatedly deleting a valid arc gives a sequence of OSOS-free networks that must terminate at a tree, as~$N$ is a finite graph.
\end{proof}

For trees, tree-child, and orchard networks, we have the OSOS characterization for determining whether $A^*_{\cC}(N)$ is finite, which follows directly from~\Cref{lem:At_Aor_finite}, \Cref{thm:FiniteLTiffOSOS-free}, and \Cref{obs:A*ClassIneq}. 
\ym{\begin{corollary}\label{cor:OSOSfImpliesFinite}
    Let~$N$ be a network. Let~$\cC\in \{\cT,\cTC,\cOr,\cTB\}$. If~$N$ is OSOS-free, then $A^*_{\cC}(N)$ is finite.
\end{corollary}}

For the classes of tree-child and orchard, we also have the converse.

\begin{corollary}\label{cor:FiniteImpliesOSOSTCO}
    Let~$N$ be a network. Let~$\cC\in \{\cT,\cTC,\cOr\}$. 
    If~$A^*_{\cC}(N)$ is finite, then~$N$ is OSOS-free.
\end{corollary}
\begin{proof}
    For trees, the claim follows from \Cref{thm:FiniteLTiffOSOS-free}.
    If~$N$ contains an OSOS-subgraph, then it must contain an omnian, so it cannot be tree-child.
    If~$N$ contains an OSOS-subgraph, then it is easy to check that any cherry cover is cyclic (there is no way to reduce an OSOS-subgraph by cherry-picking).
    Reducing OSOS-subgraphs using vaid arc deletions is impossible.
    In the process, one obtains an OSOS-subgraph with two reticulations (or more), wherein all reticulation arcs are invalid.
\end{proof}

\ym{Note that the converse of \Cref{cor:OSOSfImpliesFinite} for tree-based networks is not always true. 
Networks containing OSOS-subgraphs can still have a finite~$A^*_{\cTB}$ score (see \Cref{subfig:osos_tb}).
This is because in general, networks with an OSOS-subgraph can be tree-based.
Of course, there are also numerous networks that contain OSOS-subgraphs which are non-tree-based, which have undefined~$A^*_{\cTB}$ scores (see~\Cref{fig:tulip_field}).}

\begin{figure}
    \centering
    \subfloat[]{\includegraphics[height=6cm]{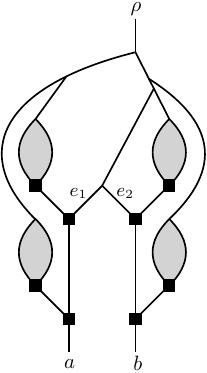} \label{subfig:tulip_N}}
    \hspace{2mm}
    \subfloat[]{\includegraphics[height=6cm]{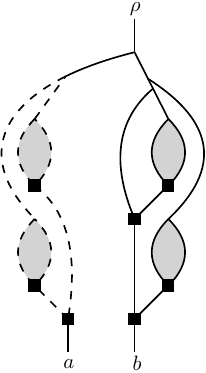} \label{subfig:tulip_del}}
    \caption{(a) A non-tree-based network~$N$ where~$A^*_{\cTB}(N)$ is undefined.  
    The gray oval shapes are used to illustrate arbitrary OSOS-subgraphs. 
    The only valid arcs are~$e_1$ and~$e_2$. (b) The network obtained by deleting~$e_1$. This graph contains an OSOS-subgraph (the dashed subgraph).}
    \label{fig:tulip_field}
\end{figure}

\paragraph{Complexity.}
Next, we show that determining if a network is OSOS-free can be done in polynomial time.
A \emph{dominator tree} of a given network is a directed tree on the same set of vertices as~$N$, with an arc~$uv$ if~$u$ is an immediate dominator of~$v$. 
Such a tree can be obtained with the \emph{Lengauer-Tarjan algorithm (LT)}, which takes~$O(|E|\log(|V|))$ time for an input DAG~$(V,E)$ with a single source \cite{lengauer1979fast}.
We start with the pseudo-code.
\medskip

\begin{algorithm}[H]
 \KwData{A network~$N$ on a set of leaves~$X$.}
 \KwResult{The set of all vertex pairs that induce an OSOS-subgraph in $N$.}
 Initialize a set $T:=\emptyset$\;
 Compute the dominator tree of~$N$ using LT. Call this tree~$T^{Dom}_N$\;\label{line:DomN}
 Reverse arc directions in~$N$, add vertex~$s$, and add arcs~$sx$ for every~$x\in X$. Call this DAG~$N_R$\;
 Compute the dominator tree of~$N_R$ using LT. Call this tree~$T^{Dom}_{N_R}$\;\label{line:DomNR}
 \ForEach(){tree vertex~$x$ in~$N$\label{line:xtree}}{
 \If{$x$ has an outgoing reticulation neighbour~$y$ in~$T^{Dom}_N$\label{line:xyEdge}}{
 \If{$x$ is an outgoing neighbour of~$y$ in~$T^{Dom}_{N_R}$\label{line:yxEdge}}{
 $T = T\cup \{xy\}$\;
 }
 }
 } 
 \Return $T$\;
\caption{{\sc FindOSOS}$(N)$}\label{alg:FindOSOS}
\end{algorithm}

\begin{theorem}\label{thm:FindOSOSCorrect}
Let~$N$ be a network on~$n$ vertices. Then~{\sc FindOSOS}$(N)$ is the set of all vertex pairs that induce an OSOS-subgraph in~$N$.
In other words, \Cref{alg:FindOSOS} is correct.
The algorithm runs in~$O(n\log n)$ time.
\end{theorem}
\begin{proof}
    We first show correctness of \Cref{alg:FindOSOS}. To do so, we show that a pair of vertices~$u,v$ induces an OSOS-subgraph in~$N$ if and only if~$uv\in\text{{\sc FindOSOS}}(N)$.
    We use the notation of \Cref{alg:FindOSOS}. 
    Let~$N_R$ denote the DAG obtained by reversing all arc directions in~$N$, adding a vertex~$s$, and adding arcs~$sx$ for every~$x\in X$.
    We observe that~$u,v$ induces an OSOS-subgraph in~$N$ if and only if 
    \begin{itemize}
        \item $u$ is a tree vertex and~$v$ a reticulation in~$N$, and
        \item every path from the root to~$v$ in~$N$ must contain~$u$, and
        \item every path from~$u$ to a vertex in~$X$ in~$N$ must contain~$v$, and
        \item $u$ is the lowest vertex that satisfies the above two conditions.
    \end{itemize}
    This is true if and only if
    \begin{itemize}
        \item $u$ is a tree vertex and~$v$ is a reticulation in~$N$, and
        \item $u$ immediately dominates~$v$ in~$N$ and $v$ immediately dominates~$u$ in~$N_R$.
    \end{itemize}
    In the notation of the pseudocode, this is true if and only if 
    \begin{itemize}
        \item $u$ is a tree vertex (\Cref{line:xtree}) and~$v$ is a reticulation (\Cref{line:yxEdge}) in~$N$, and
        \item $v$ is an outgoing neighbour of~$u$ in~$T_N^{Dom}$ (\Cref{line:xyEdge}), and
        \item $u$ is an outgoing neighbour of~$v$ in~$T_{N_R}^{Dom}$ (\Cref{line:yxEdge}).
    \end{itemize} 
    Dominators in acyclic directed graphs can be obtained using the Lengauer-Tarjan algorithm.
    Through the chain of equivalent statements, the claim follows immediately.
    \medskip
    
    To see the running time, observe that in a binary network, $|E| = O(|V|)$, and so the Lengauer-Tarjan algorithm runs in~$O(n\log n)$ time.
    The algorithm is invoked separately on~\Cref{line:DomN,line:DomNR}.
    The for loop of \Cref{line:xtree} is iterated at most~$O(n)$ times, and \Cref{line:xyEdge,line:yxEdge} can be done in constant time.
    Thus the algorithm runs in~$O(n\log n)$ time.
\end{proof}

\section{Hardness Proofs for Orchard Networks}
\label{sec:Hardness}

\subsection{Leaf Addition Measure}\label{subsec:LeafAdditionHard}
In this section, we show that computing~$L_\cOr(N)$ is NP-hard by reducing from  
degree-3 vertex cover.

\medskip
\noindent
\begin{center}
\fbox{\parbox{0.9\linewidth}{
{\sc \dpVC{} (Decision)}\\
{\bf Input:} A 3-regular graph $G = (V, E)$ and a natural number $k$.\\
{\bf Decide:} Does G have a vertex cover with at most $k$ vertices?}}
\end{center}
\noindent
\begin{center}
\fbox{
\parbox{0.9\linewidth}{
{\sc \dpLD{} (Decision)}\\
{\bf Input:} A network $N$ on a set of taxa $X$ and a natural number $k$.\\
{\bf Decide:} Can $N$ be made orchard with at most $k$ leaf additions?}
}
\end{center}

\begin{figure}
    \centering    
    \includegraphics[width=\columnwidth]{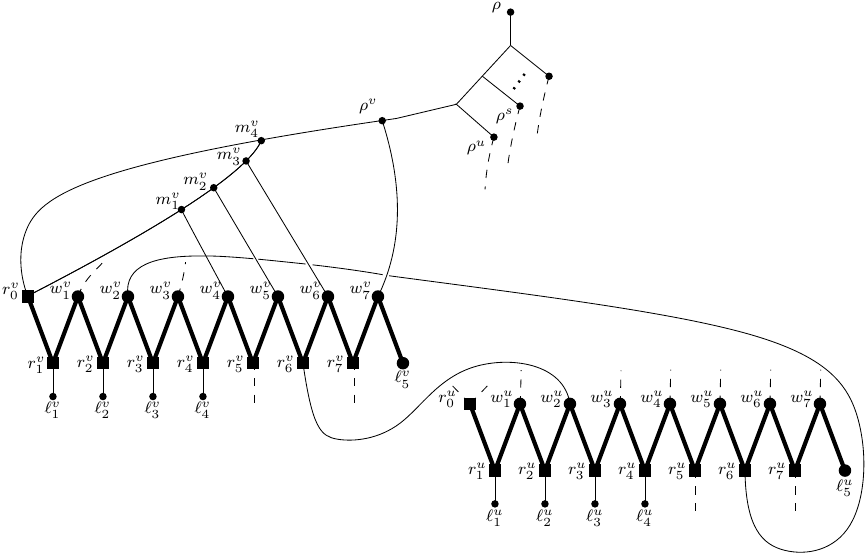}
    \caption{Sketch of the network~$N_G$ for the case when~$G$ contains an arc~$uv$.} 
    \label{fig:VCtoLORDreduction}
\end{figure}

\medskip
We now describe the reduction from \dpVC{} to \dpLD{}. For a graph~$G$, let~$V(G)$ and~$E(G)$ be its vertex and arc sets, respectively. Given an instance~$(G,k)$ of \dpVC, construct an instance~$(N_G,k)$ of \dpLD{} as follows (see \Cref{fig:VCtoLORDreduction}):

\begin{enumerate}
    \item For each vertex~$v$ in $V(G)$, construct a gadget~$\Gad(v)$ as described below.
    In what follows, vertices of the form $\ell_i^v$ are leaves, vertices $r_i^v$ are reticulations, and vertices $w_i^v$, $m_i^v$ and $\rho^v$ are tree vertices.
    
    The key structure in~$\Gad(v)$ is an N-fence with $15$ arcs, starting with the arc 
    $\boldsymbol{r_0^vr_1^v}$, then followed by arcs of the form $\boldsymbol{w_i^vr_i^v, w_i^vr_{i+1}^v}$ for each $i\in[6]$, and finally the arcs $\boldsymbol{w_7^vr_7^v,w_7^v\ell_5^v}$. This set of arcs, in bold type, is called the \emph{principal part} of~$\Gad(v)$. 
    In addition,the reticulations $r_1^v, r_2^v,r_3^v,r_4^v$ have leaf children $\ell_1^v,\ell_2^v, \ell_3^v, \ell_4^v$ respectively.
    
    Above the principal part of~$\Gad(v)$, add a set of tree vertices $m_1^v,m_2^v,m_3^v,m_4^v,\rho^v$ with the following children: $m_1^v$ has children $r_0^v$ and $w_4^v$, $m_2^v$ has children $m_1^v$ and $w_5^v$, $m_3^v$ has children $m_2^v$ and $w_6^v$, $m_4^v$ has children $m_3^v$ and $r_0^v$, and $\rho^v$ has children $m_4^v$ and $w_7^v$ (see \Cref{fig:VCtoLORDreduction}).
    
    This completes the construction of~$\Gad(v)$. Note that so far, the vertices $w_1^v,w_2^v,w_3^v$ have no incoming arcs, and $r_5^v,r_6^v,r_7^v$ have no outgoing arcs. Such arcs will be added later to connected different gadgets together.

    \item Connect the vertices $\rho^v$ from each~$\Gad(v)$ as follows: take some ordering of the vertices $\{v_1,\ldots, v_g\}$ of~$G$. Add a vertex~$\rho$ and vertices~$s_i$ for~$i\in [g-1]$. Add arcs~$\rho s_1$ and also arcs from the set~$\{s_is_{i+1}: i\in [g-2]\}$, as well as arcs from the set~$\{s_i\rho^{v_i}: i\in [g-1]\}$, and finally an arc~$s_{g-1}\rho^{v_g}$.
    
    \item Next add arcs between the gadgets corresponding to adjacent vertices in $G$, in the following way:
    for every pair of adjacent vertices $u,v$ in $G$, add an arc connecting one of the vertices $r_5^u,r_6^u,r_7^u$ in $\Gad(u)$ to one of the vertices $w_1^v,w_2^v,w_3^v$ in $\Gad(v)$  (and, symmetrically, an arc connecting one of $r_5^v,r_6^v,r_7^v$ to one of $w_1^u,w_2^u,w_3^u$). The exact choice of vertices connected by an arc does not matter, except that we should ensure each vertex is used by such an arc exactly once.
    Formally: for each vertex $v$ in $G$ with neighbours $a,b,c$, fix two (arbitrary) bijections $\pi_v:\{a,b,c\} \rightarrow \{1,2,3\}$ and $\tau_v:\{a,b,c\}\rightarrow \{5,6,7\}$. Then for each pair of adjacent vertices $u,v$ in $G$, add an arc from $r_{\tau_u(v)}^u$ to $w_{\pi_v(u)}^v$ (and, symmetrically, add an arc from $r_{\tau_v(u)}^v$ to $w_{\pi_u(v)}^u$).

    \item Finally, for each vertex~$v$ in~$G$, label the vertices~$\{\ell_i^v: i\in[5]\}$ in~$\Gad(v)$ by~$\ell_i^v$.
\end{enumerate}

Call the resulting graph~$N_G$.

\begin{lemma}\label{lem:N_Gisnetwork}
    The graph $N_G$ is a phylogenetic network on the leaf set~$\{\ell_i^v:i\in[5] \text{, $v\in V(G)$}\}.$
\end{lemma}
\begin{proof}
It is easy to see that~$N_G$ is directed. To see that it is acyclic, we give a topological order on the vertices.

For each vertex~$v\in V(G)$, we partition and order the vertices of $\Gad(v)$ in~$N_G$ as follows.
We shall write $V_1(\Gad(v)) = (\rho^{v}, m^{v}_4, m^{v}_3, m^{v}_2, m^{v}_1, r^{v}_0, w^{v}_4, w^{v}_5, w^{v}_6, w^{v}_7, r^{v}_5, r^{v}_6, r^{v}_7, \ell^{v}_5)$
and 
$V_2(\Gad(v)) = (w^{v}_1, w^{v}_2, w^{v}_3, r^{v}_1, r^{v}_2, r^{v}_3, r^{v}_4, \ell^{v}_1, \ell^{v}_2, \ell^{v}_3, \ell^{v}_4)$.
Locally, it is easy to see that both vertex sets $V_1(\Gad(v))$ and $V_2(\Gad(v))$ give partial orders.
We give another partial order on the spine of the tree that connects the root of all gadgets. We write $S = (\rho, s_1, s_2, \ldots, s_{g-1})$. 
Note also that $S$ gives a partial order.

We are now ready to concatenate sequences to give a topological ordering on~$N_G$.
Given two sequences~$x=(x_1,\ldots, x_a)$ and~$y = (y_1,\ldots, y_b)$, we write $x\frown y = (x_1,\ldots, x_a, y_1,\ldots, y_b)$ to denote its concatenation.
We claim that $s\frown V_1(\Gad(v_1))\frown \cdots \frown V_1(\Gad(v_g)) \frown V_2(\Gad(v_1))\frown \cdots \frown V_2(\Gad(v_g))$ is a topological ordering on $N_G$.
Indeed, by construction, we observe that no vertices of other gadgets are above any vertices of $V_1(\Gad(v))$, for any~$v\in V(G)$. 
The only vertices above this part of the gadget are included in the sequence~$S$.
In addition, no vertices of~$V_2(\Gad(v))$ are above vertices of other gadgets, for all~$v\in V(G)$.
Therefore, $N_G$ is acyclic, and thus it is a DAG
with a single root~$\rho$. 
All leaves are labelled, and all internal vertices are either tree vertices or reticulations.
Therefore it is a network on the leaf set~$\{\ell_i^v:i\in[5] \text{, $v\in V(G)$}\}.$
\end{proof}

\markj{As the arcs of~$N_G$ are decomposed into M-fences and N-fences, we have the following observation.}

\begin{observation}\label{obs:N_GisTB}
    Let~$G$ be a 3-regular graph and let~$N_G$ be the network obtained by the reduction. Then~$N_G$ is tree-based.
\end{observation}
\begin{proof}
    It is easy to check that the arcs of~$N_G$ are decomposed into M-fences and N-fences.
    \ym{Indeed,} the principal part of each gadget~$\Gad(v)$ is an N-fence; each arc leaving the principal part of a gadget~$\Gad(v)$ is an N-fence of length $1$; the remaining arcs decompose into M-fences of length $2$. By \Cref{lem:TB=NoW},~$N_G$ must be tree-based.  
\end{proof}

By \Cref{lem:N_Gisnetwork}, \Cref{obs:N_GisTB}, and \Cref{thm:TB=ChCover}, we use freely from now on that~$N_G$ has a cherry cover.
Before proving the main result, we require some notation and helper lemmas. 
Let~$N$ be a network and let~$\hat{N_i}$ be an N-fence of~$N$. In what follows, we shall write~$\hat{N_i} := (a^i_1,a^i_2,\ldots, a^i_{k_i})$, and we will let~$c^i_{2j-1}$ denote the child of~$\head(a^i_{2j-1})$ for~$j\in\left[\frac{k_i-1}{2}\right]$.
The first lemma states that although a tree-based network may have non-unique cherry covers, the reticulated cherry shapes that cover arcs of N-fences are fixed.

\begin{lemma} \label{lem:n_fence_unique}
    Let~$N$ be a tree-based network, and let~$\hat{N_1},\hat{N_2},\ldots, \hat{N_n}$ denote the N-fences of~$N$ of length at least 3. 
    Then every cherry cover of~$N$ must contain the reticulated cherry shapes $\{(\head(a^i_{2j-1}),c^i_{2j-1}), a^i_{2j},a^i_{2j+1}\}$ for~$i\in[n]$ and~$j\in\left[\frac{k_i-1}{2}\right]$.
\end{lemma}
\begin{proof}
    Let~$\hat{N_i} = (a^i_1,a^i_2,\ldots, a^i_{k_i})$ be an N-fence of length~$k_i\ge 3$.
    Observe that in every cherry cover, exactly one incoming arc of every reticulation is covered by a reticulated cherry shape as a middle arc (since the network is binary; for non-binary networks, this is not true in general~\cite{van2021unifying}).
    Since~$\head(a^i_1)$ is a reticulation, one of~$a^i_1$ or~$a^i_2$ must be in a reticulated cherry shape as a middle arc. But~$\tail(a^i_1)$ is a reticulation; therefore,~$a^i_2$ must be in a middle arc of a reticulated cherry shape.
    The other two arcs of the same reticulated cherry shapes are then fixed to be~$\head(a^i_1)c^i_1$ and~$a^i_3$. 
    Repeating this argument for the reticulations~$\head(a^i_{2j+1})$ for~$j\in\left[\frac{k_i-1}{2}\right]$ gives the required claim for the N-fence~$\hat{N_i}$; further repeating this argument for every N-fence gives the required claim.
\end{proof}

Note that the principal part of a gadget~$\Gad(v)$ for every~$v\in V(G)$ is an N-fence. Let us denote the principal part of a gadget~$\Gad(v)$ by~$(a^v_1,a^v_2,\ldots,a^v_{15})$ for all~$v\in V(G)$. By \Cref{lem:n_fence_unique}, $a^v_i$ for~$i=2,\ldots, 15$ and the outgoing arcs of~$r_i^v$ are covered in the same manner across all possible cherry covers of~$N_G$. Let us denote the reticulated cherry shape that contains~$a^v_i$ and~$a^v_{i+1}$ by~$R^v_{i/2}$ for even $i \in [15]$. 
\Cref{subfig:gadgets_cherry_cover,subfig:gadgets_cherry_cover_graph} show an example of the part of cherry cover auxiliary graph containing~$R^v_{i}$ and~$R^u_i$ for~$i\in[7]$, for some arc~$uv$ in $G$.
Note that the cherry shapes form a cycle.
The next lemma implies that in fact,  such a cycle exists for any arc~$uv$ in $G$.
The proof of the following is very similar to the proof of \Cref{lem:n_fence_unique}.

\begin{figure}
    \centering
    \subfloat[]{\includegraphics[height=4cm]{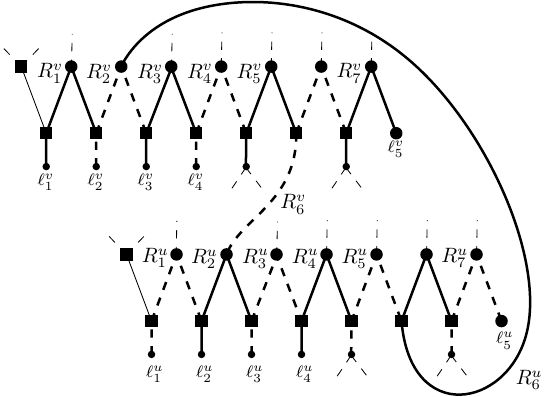} \label{subfig:gadgets_cherry_cover}}
        \subfloat[]{\includegraphics[height=4cm]{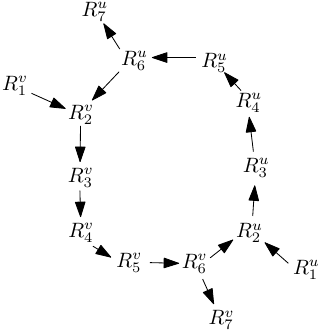}\label{subfig:gadgets_cherry_cover_graph}}
        \\
    \subfloat[]{\includegraphics[height=4cm]{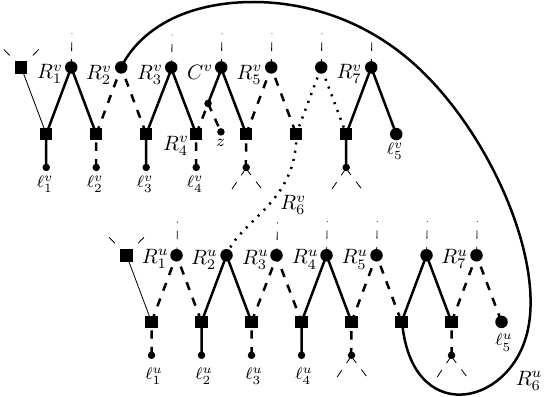}\label{subfig:gadgets_cherry_cover_leaf}}
        \subfloat[]{\includegraphics[height=4cm]{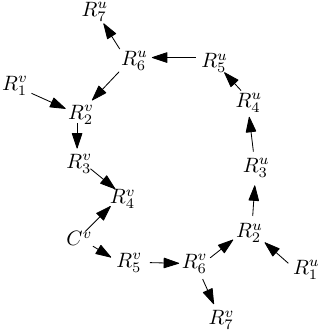}\label{subfig:gadgets_cherry_cover_leaf_graph}}
    \caption{Cherry cover of~$\Gad(v)$ and~$\Gad(u)$. In (a), the unique cherry cover of the principal part of~$\Gad(v)$ and~$\Gad(u)$ is displayed, in (b), the cherry cover auxiliary graph of (a) is given. In (c), the leaf~$z \notin X$ is added to the principal part of~$\Gad(v)$, and one possible cherry cover of the same part of the network is given. And in (d), the cherry cover auxiliary graph of (c) is given.}
    \label{fig:reduction_cherry_cover}
\end{figure}

\begin{lemma}
\label{lem:n_fence_cycle}
Let $N$ be a tree-based network and suppose that for two N-fences 
$\hat{N_u} := (a^u_1,a^u_2,\ldots, a^u_{k_u})$ and 
$\hat{N_v} := (a^v_1,a^v_2,\ldots, a^v_{k_v})$ of length at least $3$, there exist directed paths in $N$ from $\head(a^u_h)$ to $\tail(a^v_i)$ and from $\head(a^v_j)$ to $\tail(a^u_k)$, for even $h,i,j,k$ with $k < h$ and $i < j$. 
Then every cherry cover auxiliary graph of~$N$ contains a cycle.
\end{lemma}
\begin{proof}
    Let us again denote the reticulated cherry shape that contains~$a^u_h$ and~$a^u_{h+1}$ by~$R^u_{h/2}$, and similarly for~$R^v_{i/2}$,~$R^v_{j/2}$,and~$R^u_{k/2}$.
    By \Cref{lem:n_fence_unique}, all of ~$R^u_{h/2}$,~$R^v_{i/2}$,~$R^v_{j/2}$,~$R^u_{k/2}$ appear in the cherry cover auxiliary graph. Moreover~$R^u_{k/2}$ is above~$R^u_{h/2}$, and~$R^v_{i/2}$ is above~$R^v_{j/2}$.
    Now observe that for any consecutive arcs on the path from $\head(a^u_h)$ to $\tail(a^v_i)$, either they are part of the same reticulated cherry shape in the cherry cover, or they are part of different cherry shapes with one cherry shape directly above the other. This implies that there is a path from~$R^u_{h/2}$ to~$R^v_{i/2}$ in the cherry cover auxiliary graph. A similar argument shows that there is a path from~$R^v_{j/2}$ to~$R^u_{k/2}$. 
    But then we have that~$R^u_{h/2}$ is above~$R^v_{i/2}$, which is above~$R^v_{j/2}$, which is above~$R^u_{k/2}$, which is above~$R^u_{h/2}$ and we have a cycle.
\end{proof}

\ym{Orchard networks admit acyclic cherry covers.}
\markj{In order to remove all possible cycles from a possible cherry cover, it is therefore necessary to disrupt the principal part of either $\Gad(u)$ or $\Gad(v)$, for any arc~$uv$ in $G$.}

\begin{lemma}\label{lem:LORDimpliesVC}
    Let~$G$ be a 3-regular graph and let~$N_G$ be the network obtained by the reduction above. Suppose that~$A$ is a set of arcs of~$N_G$, for which adding leaves to every arc in~$A$ results in an orchard network.  
    For every arc~$uv\in E(G)$, there exists an arc~$a\in A$ that is an arc of the principal part of~$\Gad(u)$ or~$\Gad(v)$.
\end{lemma}
\begin{proof}
    We prove this lemma by contraposition. Let us assume that there is an arc~$uv\in E(G)$, such that no arcs of the principal part of~$\Gad(u)$ or~$\Gad(v)$ are in~$A$.
    We shall show that the network obtained by adding leaves to all~$a \in A$ in $N_G$ -- which we denote~$N_G+A$ -- is not orchard.
  
    From \Cref{thm:AcyclicCherryCover} we know that $N_G$ is orchard if and only if~$N_G$ has an acyclic cherry cover.
    We show here that~$N_G+A$ will not have an acyclic cherry cover, thereby showing that~$N_G+A$ is not orchard. 

    As no arcs were added to the principal part of~$\Gad(u)$ or~$\Gad(v)$, these principal parts remain N-fences in $N_G+A$. Furthermore by construction, $N$ has an arc from some~$\head(a_h^v)$ to~$\tail(a_i^u)$ for even $h \geq 10$ and even $i \leq 6$, and so~$N_G+A$ has a path from~$\head(a_h^v)$ to~$\tail(a_i^u)$. Similarly ~$N_G+A$ has a path from~$\head(a_j^u)$ to~$\tail(a_k^v)$ for some even $j \geq 10$ and $k\leq 6$. Then \Cref{lem:n_fence_cycle} implies that the auxiliary graph of any  cherry cover of  $N_G+A$ contains a cycle. By \Cref{thm:AcyclicCherryCover}, we have that~$N_G+A$ is not orchard.  
\end{proof}

\markj{To complete the proof of the validity of the reduction, we show that in order to make $N_G$ orchard by leaf additions, it is sufficient (and necessary) to add a leaf~$z^v$ to an appropriate arc of~$\Gad(v)$ for every $v$ in a vertex cover~$V_{sol}$ of $G$ (see \Cref{subfig:gadgets_cherry_cover_leaf}). 
    The key idea is that this splits the principal part of $\Gad(v)$ from an N-fence into an N-fence and an M-fence, and this allows us to avoid the cycle in the cherry cover auxiliary graph (see \Cref{subfig:gadgets_cherry_cover_leaf_graph}).}

\begin{lemma}\label{lem:VCiffLORD}
    Let~$G$ be a 3-regular graph and let~$N_G$ be the network obtained by the reduction described above. Then $G$ has a minimum vertex cover of size at most~$k$ if and only if~$L_{\cOr}(N_G) \le k$.
\end{lemma}
\begin{proof}
    Suppose first that~$V_{sol}$ is a vertex cover of~$G$ with at most~$k$ vertices. We shall show that adding a leaf to an arc of the principal part of each~$\Gad(v)$ for~$v \in V_{sol}$ makes $N_G$ orchard. 
    This will show that the minimum vertex cover of~$G$ is at least $L_{\cOr}(N_G)$. In the remainder of this proof, we will refer to vertices and arcs of~$N_G$ as introduced above in the reduction.
    
    For every~$v \in V_{sol}$, we add a leaf~$z^v$ to the arc~$w_4^v r_4^v$ of~$\Gad(v)$ (see \Cref{subfig:gadgets_cherry_cover_leaf}). Let~$q^v$ be the parent of~$z^v$.
    This splits the principal part of $\Gad(v)$ from an N-fence into an N-fence and an M-fence, and this allows us to avoid the cycle in the cherry cover auxiliary graph (see \Cref{subfig:gadgets_cherry_cover_leaf_graph}).
    
    Let us call the new network~$M$.
    To formally show that $M$ is orchard, we give an HGT-consistent labelling~$t:V(M)\rightarrow\R$.
    
    Begin by setting $t(\rho) = 0$, and for any vertex in $s_1,\dots, s_{g-1}$ or $\rho^v, m_4^v,\dots, m_2^v$ for any $v$ in $V(G)$, let this vertex have label equal to the label of its parent plus~$1$. 
    Let $h$ be the maximum value assigned to a vertex so far, and now adjust $t$ by subtracting $(h+1)$ from each label. Thus, we may now assume that all vertices in $\rho, s_1,\dots, s_{g-1}$ or $\rho^v, m_4^v,\dots, m_2^v$ for any $v$ in $V(G)$ have label $\leq -1$. Now set  $t(m_1^v) = 0$ and $t(r_0^v) = 0$, for each $v$ in $V(G)$.
    
    It is easy to see that so far $t$ satisfies the properties of an HGT-consistent labelling. It remains to label the vertices in the principal part of each gadget~$\Gad(v)$, and the leaves of each gadget, and the new vertices $q^v$ and $z^v$ for $v \in~V_{sol}$. We do this as follows.
    
    For $v \in~V_{sol}$, set $t(r_1^v)=t(w_1^v) = 12, t(r_2^v)=t(w_2^v) = 13, t(r_3^v)=t(w_3^v) = 14$, and $t(r_4^v) = t(q^v) = 15$. 
    Set $t(w_4^v) = 1, t(r_5^v) = t(w_5^v)=2, t(r_6^v) = t(w_6^v) = 3$, and $t(r_7^v) = t(w_7^v) = 4$.
    
    For $v \notin~V_{sol}$, set $t(r_1^v) = t(w_1^v) = 5$, and $t(r_i^v)=t(w_i^v) = i+4$ for every $i$ up to $t(r_7^v) = t(w_7^v) = 11$.
    
    Finally, for each leaf $\ell$ with parent $p$ set $t(\ell) = t(p) + 1$.
    
    It remains to observe that $t$ is a non-temporal labelling of $M$ and for every reticulation $r$ in $M$, $r$ has exactly one parent $p$ with $t(p)=t(r)$. Thus $t$ is an HGT-consistent labelling of $M$, and it follows from~\Cref{thm:OrchIFFHori} that $M$ is orchard.  
    
    
    Suppose now that we have a set of arcs~$A_{sol}$ of~$N_G$ of size at most~$k$, such that adding leaves to the arcs in~$A_{sol}$ makes~$N_G$ orchard. By \Cref{lem:LORDimpliesVC}, for every arc~$uv\in E(G)$, there exists an arc~$a\in A_{sol}$ that is an arc of the principal part of~$\Gad(u)$ or~$\Gad(v)$. 
    It follows immediately that the set~$\{v\in V(G): \text{$A_{sol}$ contains an arc of the principal part of~$\Gad(v)$}\}$ is a vertex cover of~$G$. 
    Since this is true for any such set of arcs~$A_{sol}$, it follows that if there is such an~$A_{sol}$ of size at most~$k$, then there must exist a vertex cover of~$G$ of size at most~$k$. 
\end{proof}

\begin{theorem}\label{thm:L_Or=Hard}
    Let~$N$ be a network. The decision problem \dpLD{} is NP-complete. Computing $L_\cOr(N)$ is NP-hard. 
\end{theorem}

\begin{proof}
    Suppose we are given a set of arcs~$A_{sol}$ of~$N_G$ of size at most~$k$.
    Upon adding leaves to every arc in~$A_{sol}$, we may check that the resulting network is orchard in polynomial time (see Section 6 of \cite{janssen2021cherry}).
    This implies that~\dpLD{} is in NP. The reduction from \dpVC{} to \dpLD{} outlined at the start of the section takes polynomial time, since we add a constant number of vertices and arcs for every vertex in the \dpVC{} instance. The NP-completeness of \dpLD{} follows from \Cref{lem:VCiffLORD}. The optimization problem of \dpLD{}, i.e., the one of computing~$L_\cOr(N)$ is therefore NP-hard.
\end{proof}

\subsection{Valid Arc Deletion Measure}\label{subsec:ArcDelHard}
\ym{In this subsection, we show that the decision problem variant of computing $A^*_{\cOr}$ is also NP-complete, using the same reduction as outlined in \Cref{subsec:LeafAdditionHard}.}
\ej{Recall that the order in which valid arcs are deleted matters (see \Cref{sec:A*Deletion}). 
Therefore, we consider a sequence~$A = (a_1,\ldots, a_k)$ of valid arcs, such that $a_i$ is a valid arc in the network obtained by removing $(a_1,\ldots, a_{i-1})$ sequentially from the original network.}

\medskip
\begin{center}
\noindent\fbox{\parbox{0.9\linewidth}{
{\sc \dpAsD~(Decision)}\\
{\bf Input:} A network $N$ on $X$ and a natural number $k$.\\
{\bf Decide:} Is $N$ orchard after at most $k$ valid arc deletions?}}
\end{center}
\medskip



\begin{lemma}\label{lem:AORDimpliesVC}
    Let~$G$ be a 3-regular graph and let~$N_G$ be the network obtained by the reduction in \Cref{subsec:LeafAdditionHard}. Suppose that~$A$ is a sequence of valid arcs, such that deleting them in order results in an orchard network.  
    For every arc~$uv\in E(G)$, there exists an arc~$a\in A$ that is an arc of~$\Gad(u)$ or~$\Gad(v)$.
\end{lemma}
\begin{proof}
    We give a similar proof as done for \Cref{lem:LORDimpliesVC}. We prove by contraposition.
    Suppose there is an arc~$uv\in E(G)$ such that no arcs of~$\Gad(u)$ and~$\Gad(v)$ are in~$A$.
    We show that the network obtained after sequentially deleting $A$ is non-orchard.

    For an arc to be valid, its head must be a reticulation. 
    Therefore, the only arcs we may delete are contained in gadgets, namely, those with $r_i^w$ as its head, for $i\in [7]$, for all vertices~$w\in V(G)$.
    Every arc that feeds out of the gadget~$\Gad(u)$ has a reticulation $r^v_5, r^v_6, r^v_7$ as its tail; every arc that feeds in to $\Gad(u)$ has a tree vertex~$w^v_1,w^v_2, w^v_3$ as its head.
    By assumption, no arcs of~$\Gad(u)$ and $\Gad(v)$ are in~$A$.
    Due to these facts, no arc in~$A$ is incident with a vertex in~$\Gad(u)$ or~$\Gad(v)$.
    Let~$N'$ denote the network obtained by sequentially deleting arcs of~$A$ from~$N_G$. 
    In~$N'$, the gadgets~$\Gad(u)$ and~$\Gad(v)$ remain intact, with regards to its vertices, incoming arcs, and outgoing arcs (the vertices incident to the incoming and outgoing arcs may be different).

    In particular, the principal parts of~$\Gad(u)$ and $\Gad(v)$ remain N-fences in~$N'$. 
    Similarly as in the proof of \Cref{lem:LORDimpliesVC}, one can verify that the conditions for \Cref{lem:n_fence_cycle} are met.
    This means that the auxiliary graph of any cherry cover of~$N'$ contains a cycle.
    By \Cref{thm:AcyclicCherryCover}, we have that~$N'$ is not orchard.
\end{proof}

\begin{lemma}\label{lem:VCiffAORD}
    Let~$G$ be a 3-regular graph and let~$N_G$ be the network obtained by the reduction described above. Then $G$ has a vertex cover of size at most~$k$ if and only if~$A^*_{\cOr}(N_G) \le k$.
\end{lemma}
\begin{proof}
    We argue analogously as done in the proof of \Cref{lem:VCiffLORD}.
    Let~$V_{sol}$ be a vertex cover of~$G$ of size at most~$k$.
    For every vertex~$v\in V_{sol}$, we remove an arc $w_4^v r_4^v$ from each~$\Gad(v)$. 
    Note that for each of the deleted arc, the two endpoints are also suppressed.
    Observe that these selected arcs are valid in the network they are removed in, independently of the order in which the gadgets are processed.
    Call the resulting graph~$N$. We shall show that~$N$ is orchard, thereby showing that~$A^*_{\cOr}(N_G) \le k$.
    In the remainder of this proof, we will refer to vertices and arcs of~$N_G$ as introduced above in the reduction.

    To show that~$N$ is orchard, we give an HGT-consistent labelling~$t:V(N) \rightarrow \mathbb{R}$. 
    The exact same labelling as given in the proof of \Cref{lem:VCiffLORD} suffices. The only difference is that the vertex sets are different. In particular, $w_4^v$ and $r_4^v$ are no longer vertices here, and the vertices~$q^v,z^v$ do not exist for every~$v\in V_{sol}$. 

    
    
    
    
    

    To prove the other direction, suppose that deleting the arcs of some set~$A_{sol}$ sequentially from~$N_G$ results in an orchard network, where~$|A_{sol}|\le k$.
    By \Cref{lem:AORDimpliesVC}, for every arc $uv\in E(G)$, there exists an arc~$a\in A_{sol}$ that is in~$\Gad(u)$ or~$\Gad(v)$. 
    It follows immediately that the set $\{v\in V(G): A_{sol} \text{ contains an arc of~$\Gad(v)$}\}$ is a vertex cover of~$G$. Clearly, this set is of size at most~$k$, which gives the required claim.
\end{proof}

\begin{theorem}\label{thm:A*_Or=Hard}
	Let~$N$ be a network. The decision problem \dpAsD~is NP-complete. Computing $A^*_\cOr(N)$ is NP-hard. 
\end{theorem}
\begin{proof}
    To see that the problem is in NP, note that a sequence of arcs~$A$ works as a certificate.
    Upon deleting the arcs sequentially from~$A$, one can check in polynomial time if the resulting network is orchard (see Section 6 of \cite{janssen2021cherry}).
    The reduction outlined at the start of \Cref{sec:Hardness} can be carried out in polynomial time.
    By \Cref{lem:VCiffAORD}, an input graph~$G$ is a \textsc{YES}-instance of \textsc{Degree-3 Vertex Cover} if and only if $N_G$ is a \textsc{YES}-instance of \dpAsD.
    This completes our NP-completeness proof. It follows immediately that computing $A^*_\cOr(N)$ is NP-hard.
\end{proof}

\section{Discussion}\label{sec:discussion}

\ym{We considered three types of proximity measures --- based on leaf addition, valid arc deletion, and arc deletion --- to the classes of trees, tree-child networks, orchard networks, and tree-based networks.
Given a network, we asked `how many operations (e.g., leaf addition) are needed to transform the network into one of a given class?'.
}

\ym{
\begin{table}[ht]
\centering
\begin{tabular}{l|ll|ll}
\toprule
\textbf{Class~$\cC$} & \textbf{Complexity} & \textbf{Reference} & \textbf{Upper Bound} & \textbf{Reference}\\
\midrule
Tree-child  & $O(|N|)$   & \Cref{thm:L_TCisPoly} & $\displaystyle\left\lfloor (3r-1)/2\right\rfloor$ & \Cref{thm:L_TCBound} \\
Tree-based  & $O(|N|)$   & \cite[Corollary 5.4]{hayamizu2021structure} & $\displaystyle\left\lfloor(r-1)/2\right\rfloor$ & \Cref{thm:L_TBBound}\\
Orchard     & NP-complete & \Cref{thm:L_Or=Hard} & $r-1$ & \Cref{thm:LORUpperBound} \\
\bottomrule
\end{tabular}
\caption{Computational complexity and upper bound of the leaf addition proximity measure $L_\cC(N)$ for a network~$N$.}
\label{tab:leaf-addition}
\end{table}
}
\ym{
We summarize the findings for the leaf addition measure~$L_\cC$ in \Cref{tab:leaf-addition}.
For tree-child networks and tree-based networks, one can remove local forbidden structures by adding leaves for each omnian and W-fence, respectively.
For orchard networks, this is clearly not possible, as computing the~$L_\cOr(N)$ measure is NP-hard.
This strongly suggests that a characterization of the class using local forbidden structures does not exist.}

\ym{
We also considered the pairwise comparability of the three proximity measures in \Cref{sec:comp}.
For the two arc deletions, since each valid arc deletion is an arc deletion, we have~$A_\cC\le A^*_\cC$ (\Cref{obs:AleA^*}). 
Interestingly,~$L_\cOr$ and~$A^*_\cOr$ are incomparable (\Cref{thm:IncompLorA*or}).
Similarly,~$L_\cOr$ and~$A_\cOr$ are also incomparable (\Cref{thm:IncompLorAor}).
Continuing this trend,~$L_\cTB$ and~$A_\cTB$ are incomparable (\Cref{thm:IncompLTBATB}).
Meanwhile, for all networks~$N$,~$L_\cTB(N)\le A^*_\cTB(N)$.
Intuitively, in finding a shortest sequence of valid arc deletions to obtain a tree-based network, this means that one can `mimic' each arc deletion by adding a leaf.
Moreover, for any network~$N$, we have $A_\cTC(N)\le L_\cTC(N)$ (\Cref{thm:CompLTCATC}), and it would be of great interest to know whether an analogous result holds for tree-child networks under valid arc deletions.
An example shows that~$1 = A^*_\cTC(N) \le L_\cTC(N) = 2$, and we conjecture the following.
\begin{conjecture}\label{conj:TCL_tcA*_tc}
    Let~$N$ be a network. Then~$A^*_\cTC(N)\le L_\cTC(N)$.
\end{conjecture}
}

\ym{
For valid arc deletions, we showed that if a network~$N$ does not contain an OSOS-subgraph --- which can be determined in $O(|N|\log |N|)$ time (\Cref{thm:FindOSOSCorrect}) --- then~$A^*_\cC(N)$ is finite for all classes considered in the paper (\Cref{cor:OSOSfImpliesFinite}).
The converse holds for the classes of trees, which implies that~$A^*_\cT(N)$ can be computed in polynomial time.
This is not true for other classes. For example, networks containing an OSOS-subgraph can still be tree-based (see \Cref{subfig:osos_tb}).
Even worse, computing~$A^*_\cOr(N)$ is NP-hard (\Cref{thm:A*_Or=Hard}).
The computational complexity of computing~$A^*_\cTC(N)$ and~$A^*_\cTB(N)$ remains an open question. 
\noindent
\begin{center}
\fbox{
\parbox{0.9\linewidth}{
{\textbf{Open problem 1:}}
What is the computational complexity of computing~$A^*_\cTC(N)$? What about of~$A^*_\cTB(N)$?
}
}
\end{center}
}

\ym{
The computational complexity of~$A_\cC$ also remains an open problem.
\noindent
\begin{center}
\fbox{
\parbox{0.9\linewidth}{
{\textbf{Open problem 2:}}
What is the computational complexity of computing~$A_\cC$, for~$\cC\in\{\cT,\cTC,\cOr,\cTB\}$?
}
}
\end{center}
}

\ym{
In another direction, one can also consider non-binary network inputs.
For tree-child networks and orchard networks, the results for~$L_\cC$ remain the same, as targeting omnians and the reduction from $\dpVC{}$ still hold.
For tree-based networks, however, we must consider a generalized version of the zig-zag decomposition, such as the one considered in~\cite{pons2025fence}.
In the paper, they give a characterization of semi-binary tree-based networks (tree vertices have outdegree-2 and reticulations have indegree at least 2), based on forbidden structures with this generalization.
We can of course try and target these local forbidden structures --- as we have done for W-fences --- to see if we can efficiently compute the exact~$L_\cTB$ score.
We can also pose questions on both the~$A^*_\cC$ and~$A_\cC$ measures for non-binary inputs.
}

\bibliography{z_bib}

\appendix

\end{document}